\documentclass[11pt,a4paper]{amsart}
\usepackage{graphicx} 
\usepackage{amsmath,amssymb,amscd,amsthm,hyperref,cleveref}

\usepackage[a4paper, margin=3.2cm]{geometry}

\makeatletter
\renewcommand{\l@section}{%
  \addvspace{0.6em}%
  \@tocline{1}{0pt}{1.5pc}{}{\bfseries}%
}
\renewcommand{\l@subsection}{%
  \@tocline{2}{0pt}{3.5pc}{}{\small}%
}
\makeatother

\usepackage{comment}
\usepackage{tikz}
\usetikzlibrary{shapes.geometric,positioning}
\usepackage{float}
\usepackage{mathrsfs}
\newcommand{\scr}{\mathscr}

\newcommand{\F}{\mathbb{F}}
\newcommand{\bb}[1]{\mathbb{#1}}
\newcommand{\fr}[1]{\mathfrak{#1}}
\newcommand{\ol}[1]{\overline{#1}}
\newcommand{\ca}[1]{\mathcal{#1}}
\newcommand{\ra}{\rightarrow}

\usepackage{enumitem}

\newcommand{\plot}[3]{ 
    \foreach \i in {1,...,10} {
        \pgfmathsetmacro\yval{#1[\i-1]}
        \ifnum\i<10
            \pgfmathsetmacro\nextyval{#1[\i]}
        \else
            \pgfmathsetmacro\nextyval{\yval}
        \fi
        \draw[{#3}, color=#2] (\i-0.5,\yval) -- (\i+0.5,\yval);
        \draw[{#3}, color=#2] (\i+0.5,\yval) -- (\i+0.5,\nextyval);
    }
}

\newtheorem{thm}{Theorem}[section]
\newtheorem{lemma}[thm]{Lemma}

\newtheorem{prop}[thm]{Proposition}
\newtheorem{cor}[thm]{Corollary}

\theoremstyle{definition}
\newtheorem{defi}{Definition}
\newtheorem{ex}[thm]{Example}

\theoremstyle{remark}
\newtheorem{rmk}[thm]{Remark}

\crefrangeformat{equation}{#3(#1)#4--#5(#2)#6}

\crefname{thm}{Theorem}{Theorems}
\crefname{lemma}{Lemma}{Lemmas}
\crefname{cor}{Corollary}{Corollaries}
\crefname{prop}{Proposition}{Propositions}
\crefname{defi}{Definition}{Definitions}
\crefname{ex}{Example}{Examples}
\crefname{rmk}{Remark}{Remarks}
\crefname{section}{Section}{Sections}
\crefname{equation}{\unskip}{\unskip}
\crefname{enumi}{\unskip}{\unskip}
\crefname{subsection}{Subsection}{Subsections}

\newcommand\p{\mathsf{p}}
\newcommand\q{\mathsf{q}}
\newcommand\bigp{\mathsf{P}}
\newcommand\bigq{\mathsf{Q}}
\newcommand{\aut}[1]{\mathsf{Aut}\mkern-2mu \left(#1\right)}
\newcommand{\cent}[1]{\mathsf{Z}\mkern-2mu \left(#1\right)}
\newcommand\End{\mathsf {End}}
\newcommand\Hom{\mathsf{Hom}}
\newcommand\rank{\mathsf {rank}}
\newcommand{\degg}{\operatorname{\mathsf{deg}}}
\newcommand{\Irr}{\mathsf{Irr}}
\newcommand{\Max}{\mathsf{Max}}
\newcommand{\orb}{\mathsf{Orb}}
\newcommand\chara{\mathsf {char}}
\newcommand\id{\mathsf{id}}
\newcommand\soc{\mathsf{Soc}}

\newcommand{\eval}[2]{\left. #1 \right|_{#2}}

\newcommand{\pb}[1]{\left\{ #1\right\}}
\newcommand{\seq}[1]{\left( #1\right)}

\usepackage{stmaryrd} 

\newcommand{\olfz}[1][f]{\overline{#1}\hspace{0em}^{\raisebox{0.9ex}{\tiny$z$}}}

\newcommand{\lleq}{\preccurlyeq}
\usepackage{scalerel}
\newcommand{\nback}[1][-.95pt]{%
  \mathrel{\raisebox{#1}{$\rotatebox[origin=c]{-315}{\scaleobj{0.55}{-}}$}}}
\newcommand{\llneq}{%
\mathrel{\ooalign{$\preccurlyeq$\cr\kern1.2pt$\nback$}}}

\newcounter{marg}[section]

\usepackage[normalem]{ulem}
\newcommand\rsout{\bgroup\markoverwith{\textcolor{red}{\rule[0.5ex]{2pt}{0.4pt}}}\ULon}

\begin{document}

\title{Lifting free modules to generalized Weyl algebras}
\author{Samuel A.\ Lopes}
\address{CMUP, Departamento de Matem\'atica, Faculdade de Ci\^encias, Universidade do Porto, Rua do Campo Alegre s/n, 4169--007 Porto, Portugal}
\email{slopes@fc.up.pt}
	
\author{Jonathan Nilsson}
\address{Department of Mathematics, Link\"oping University, Link\"oping, Sweden}
\email{jonathan.nilsson@liu.se}



\begin{abstract}
\noindent We study modules over a generalized Weyl algebra $R(\sigma,a)$ which are free when restricted to the base ring $R$. When $R$ is an integral domain, we construct all such finite-rank modules up to isomorphism, leading to new simple modules over a variety of algebras. In particular, we show that free modules that have rank $1$ over $R$ can be parametrized as $V_\p$ where $\p$ is a divisor of $a$. We give simplicity criteria for $V_\p$ and, additionally, when $R$ is a PID, provide a complete combinatorial description of the submodule structure of $V_\p$ and of the weight modules occurring as subquotients. We also show that, under some mild conditions on $R(\sigma,a)$, there exist simple $R$-free modules of arbitrary finite rank. We apply our results to $\fr{sl}_2$ in order to construct new families of simple Cartan-free modules of all finite ranks. 
\end{abstract}

\maketitle

\tableofcontents

\section*{Introduction}
Lie algebras and their representations play a central role in modern mathematics and mathematical physics. From the study of symmetries in quantum mechanics to the structure theory of algebraic groups, advances in representation theory have driven major developments for more than a century.

Representations of various generalizations of Lie algebras---such as deformations and quotients of enveloping algebras and quantum analogues---have also been extensively studied, often revealing structural patterns reminiscent of those found in the classical Lie setting. A notable feature of many of these algebras is that they admit concrete realizations as generalized Weyl algebras (GWAs), allowing seemingly different representation-theoretic phenomena to be approached within a single unified framework.

Generalized Weyl algebras were introduced by Bavula in his PhD thesis (see e.g.\ \cite{vB91, vB92, B93, B92, B96}) and have since been comprehensively studied. Prototypical examples of GWAs are the classical Weyl algebras (the rings of differential operators with polynomial coefficients), the enveloping algebras of the Heisenberg Lie algebra and of several low-dimensional Lie algebras, including the simple Lie algebra $\fr{sl}_2$, their primitive quotients, and many quantizations. Further examples include (generalized) noetherian down-up algebras (\cite{BR98, gB99, CS04}), noncommutative deformations of type $A$ Kleinian singularities and Smith algebras (see~\cite{spS90, tH93}), quantum Weyl algebras, and certain multiparameter quantum groups (e.g.\ \cite{QDX00, WJY02}).

Given the connections with Lie theory and quantum groups, the study of representations of GWAs has been a mainstay in the field. In particular, a theory of weight modules for GWAs has been well developed (see \cite{vB92, DGO96, BBF04}) and generalized in several directions (see e.g. \cite{LMZ15, FRS20, GRW23}). Beyond weight modules, another extensively studied class of representations is that of Whittaker modules. These were first studied for the Lie algebra $\fr{sl}_2$ in~\cite{AP74} and later systematized for complex semisimple Lie algebras in \cite{K78}. In \cite{BO09}, Benkart and Ondrus set up the framework for studying Whittaker modules for GWAs; see also \cite{CW23} for other specific examples in the class of GWAs and \cite{XZ23} for a clever adaptation to the context of quantum groups.

Recent attention has been given to a category of modules that are in a sense \textit{opposite} to the category of weight modules in Lie theory, motivated by Block's classification \cite{rB81} of simple modules for $\fr{sl}_2$. These are modules on which the enveloping algebra of a Cartan subalgebra of a semisimple Lie algebra acts freely, rather than with torsion, as is the case for weight modules. Such \emph{Cartan-free} modules were first defined by the second-named author in~\cite{jN15}, where a classification of modules of rank $1$ was also given for simple Lie algebras of type $A$. The classification was extended to type $C$ Lie algebras in~\cite{jN16}, where it was also shown that such modules do not exist for Lie algebras not of type $A$ or $C$, completing the classification of Cartan-free modules for finite-dimensional simple complex Lie algebras. Some of these modules were also independently studied in~\cite{TZ15,TZ18} in connection with Witt algebras. 

Since then, modules analogous to Cartan-free modules of rank $1$ have been studied and partially classified for a multitude of algebras, including algebras related to the Virasoro algebra~\cite{CC15, CG16}, quantum groups~\cite{XFH22}, Lie superalgebras~\cite{CZ15}, Kac-Moody algebras~\cite{CTZ20} and Smith algebras~\cite{FLM24}; we note that the techniques developed in the latter paper are closely related to those used in the present work. 

The structure of Cartan-free modules of higher rank is significantly more intricate. In~\cite{MP17}, the authors used an explicit construction to show that there exist simple Cartan-free modules of arbitrary finite rank for the Lie algebra $\fr{sl}_2$. In~\cite{GN22} this was extended by the construction of a large family of Cartan-free modules of finite rank for $\fr{sl}_n$, called tensor modules, and conditions for their simplicity were determined.

An interesting recent development is the work by Mendonça (see~\cite{M25}), where a broader class of modules was studied, namely modules which are finitely generated rather than free over $U(\fr{h})$. A key result in that paper is that any simple module that is finitely generated over $U(\fr{h})$ is either finite dimensional or Cartan-free, yielding a dichotomy between weight modules and Cartan-free modules in this setting.

We note that many of the constructions and classification results cited above for Cartan-free modules and their generalizations can be phrased in terms of representations of GWAs, and in this setting, they share a common feature: they are free over the underlying base ring $R$. This observation suggests that $R$-free modules provide a framework for understanding diverse parts of the representation theory of GWAs and the many algebras they are related to. 

In this paper we develop this perspective systematically, by constructing all $R$-free modules of finite rank for arbitrary GWAs over an integral domain $R$. It is worth noting that for a GWA, the universal Whittaker modules defined in \cite{BO09} are particular cases of the $R$-free modules which we will define herein. Our approach reveals that many previously studied modules arise naturally as special cases of our construction, and it leads to new families of simple modules simultaneously for a wide range of algebras. Moreover, for modules which are free of rank $1$ over $R$, we are able to provide simplicity criteria and completely describe composition series. In the higher rank case, we give an explicit construction of a new family of simple $R$-free modules of arbitrary finite rank. Since many algebras can be realized as GWAs, this yields several new families of simple modules over a variety of algebras.

\subsection*{Structure of the paper}
In~\cref{S:GWA} we define and collect some facts about the generalized Weyl algebra $A=R(\sigma,a)$ where $R$ is an arbitrary (noncommutative) domain. We construct and classify $R$-free modules of rank $1$ up to isomorphism, and show that such modules can be parametrized as $V_\p$ where $\p$ runs through the divisors of $a$.

In~\cref{S:UFD} we restrict ourselves to the case where $R$ is a UFD. We investigate how the $\sigma$-orbits of irreducible elements of $R$ are related to the simplicity of $V_\p$. In particular, we give a complete description of $R$-cyclic submodules of $V_\p$.

Our sharpest results are reached in~\cref{S:PID}, when $R$ is assumed to be a PID. Here we determine simplicity criteria for $V_\p$ and conditions for $V_\p$ to have finite length. When this is the case, we give a bound for the length of $V_\p$ depending only on the GWA, and we give a combinatorial algorithm to produce composition series for arbitrary $V_\p$. We show that the subquotients appearing in such composition series are weight modules, and conversely, under mild assumptions, that any abelian category containing all $R$-free modules of rank $1$ also necessarily contains all finitely supported weight modules.

\cref{S:EX} contains applications of this theory to a number of particular examples of GWAs, relating these results to some studied elsewhere.  Finally, in~\cref{SEC:highrank} we focus on $R$-free modules of arbitrary finite rank over a PID. We construct all such modules up to isomorphism and we show that (under a natural assumption) there exist simple $R$-free modules of arbitrary finite rank. We finish with the construction of novel simple Cartan-free $\fr{sl}_2$-modules of any given finite rank.

\subsection*{Conventions and notation}

Throughout the paper, $\F$ will always denote an arbitrary base field, $\bb{Z}$ is the ring of integers, $\bb{Z}_{\geq0}$ (respectively, $\bb{Z}_{>0}$) is the set of nonnegative (respectively, positive) integers. The identity map on a set $E$ is denoted by $\id_E$.

All rings, modules and homomorphisms considered will be assumed to be unital; in dealing with associative algebras, we further assume that modules and homomorphisms are unital and linear with respect to the base field. Given a ring $R$, $\cent{R}$ denotes its center and $R^\times$ its group of units.

A portion of our analysis involves (finite) multisets. When dealing with these, our convention is that unions, intersections, set differences, inclusions, products and evaluation of functions should be counted with multiplicities: $\{1,1,2\} \cup\{1,3\}=\{1,1,1,2,3\}$, $\{1,1,2\} \cap \{1,1,3\} = \{1,1\}$, $\{1,1,2\} \not\subseteq \{1,2,3\}$, $\{1,1,2\} \setminus \{1,2\} = \{1\}$, $\{1,1\} \times \{2\}=\{(1,2),(1,2)\}$, $f(\{1,1,2\})=\{f(1),f(1),f(2)\}$.

\section{$R$-free modules over generalized Weyl algebras}\label{S:GWA}

\subsection{Generalities on GWAs}

We begin with the definition of a generalized Weyl algebra (GWA for short). As with Ore extensions, this construction can of course be iterated to obtain more complex algebras (see for example~\cite{BO09}).

\begin{defi}\label{D:GWA}
For a ring $R$, an automorphism $\sigma \in \aut{R}$ and a central element $a\in \cent{R}$, the corresponding generalized Weyl algebra $R(\sigma,a)$ is generated by $R$, $x$, and $y$ subject to the relations
\begin{equation}\label{E:def:GWA}
x\sigma(r)=rx, \quad yr=\sigma(r)y, \quad xy=a, \qquad yx=\sigma(a), 
\end{equation}
for all $r\in R$.
\end{defi}

\begin{rmk}\label{R:general-facts-GWA}\hfill
\begin{enumerate}[label=(\roman*)]
\item In the literature, it is more common to find the definition of a GWA with $x$ and $y$ interchanged in the relations~\cref{E:def:GWA}. We opted for the above convention to be consistent with the usage in \cite{FLM24}. Given the symmetry of the relations, we believe that this will not confuse the reader.
\item A GWA $R(\sigma,a)$ is $\bb{Z}$-graded by putting $R$ in degree $0$, $x$ in degree $1$ and $y$ in degree $-1$. Moreover, it is both a right and left free $R$-module with basis $\pb{1, x^n, y^n\mid n\in\bb{Z}_{>0}}$. In case $R$ is an $\F$-algebra we implicitly assume that $\sigma$ is $\F$-linear and hence the corresponding GWA is an algebra over $\F$. 
\item $R(\sigma,a)$ is a (noetherian) domain if and only if $R$ is a (noetherian) domain and $a\neq 0$. 
\item $R(\sigma,a)$ is simple if and only if the following conditions hold:
\begin{enumerate}[label=(\alph*)]
    \item $a$ is not a zero divisor in $R$;
    \item $R$ has no proper nonzero two-sided $\sigma$-ideals (see \cref{D:sigma-ideal});
    \item for all $n \ge 1$, $R = Ra + R\sigma^n(a)$;
    \item no positive power of $\sigma$ is an inner automorphism of $R$.
\end{enumerate}
In case $R$ is commutative, condition~(d) simply means that $\sigma$ has infinite order. See~\cite[Theorem~4.2]{B96a} for details.\label{R:general-facts-GWA:simple}
\end{enumerate}
\end{rmk}

\textbf{Henceforth, we fix $A=R(\sigma, a)$ and assume that $A$ is a domain} i.e., that $R$ is a (not necessarily commutative) domain and that $a\neq 0$. 


Next, for the sake of completeness, we construct some key examples and address the isomorphism problem. All of these are well known and, in addition to the seminal works already cited, we refer to \cite{BJ01, RS06} for more details and further results.

\begin{ex}[The enveloping algebra of $\fr{sl}_2$]\label{Eg:sl2-as-GWA}
Consider the commutative polynomial ring $R=\F[H, C]$ with the automorphism $\sigma$ defined by $\sigma(H)=H-1$ and $\sigma(C)=C$. Take $a=C-H(H+1)$. Then, in case $\chara(\F)\neq 2$, the GWA $R(\sigma, a)$ is isomorphic to $U(\fr{sl}_2)$, the enveloping algebra of $\fr{sl}_2$. An explicit isomorphism takes the usual $\fr{sl}_2$ generators $e$, $f$, $h$ to the GWA generators $y$, $x$, $2H$. Under this isomorphism, the element $C$ in the GWA corresponds to the Casimir element $fe+\frac{1}{4}h(h+2)$ in $U(\fr{sl}_2)$.
\end{ex}

\begin{ex}[The minimal primitive quotients of $U(\fr{sl}_2$)]\label{Eg:sl2-prim-quot-as-GWA}
Consider $R=\F[h]$ and $\sigma\in\aut{R}$ defined by $\sigma(h)=h-1$, and choose any $u(h)\in\F[h]$. The GWA $R(\sigma, u)$ has been profusely studied under different names and at different times, notably in \cite{tH93} were they were related to \textit{noncommutative deformations of type-$A$ Kleinian singularities} (see also \cite{aJ77, spS90}). If $\degg u=1$, then $R(\sigma, u)$ is the Weyl algebra and if $\degg u=2$ and $\chara(\F)=0$, then $R(\sigma, u)$ is a minimal (infinite-dimensional) primitive quotient of $U(\fr{sl}_2)$. Using \cref{R:general-facts-GWA}\cref{R:general-facts-GWA:simple} we see that $R(\sigma, u)$ is simple if and only if $\chara(\F)=0$, $u\neq 0$, and there are no roots $\lambda, \mu\in\overline\F$ of $u$ with $\lambda-\mu\in\bb{Z}_{>0}$.
\end{ex}

\begin{ex}[The Smith algebras]\label{Eg:Smith-as-GWA}
The algebras \textit{similar to the enveloping algebra of $\fr{sl}_2$} introduced by Smith in~\cite{spS90} are to the \textit{noncommutative deformations of type-$A$ Kleinian singularities} in \cref{Eg:sl2-prim-quot-as-GWA} as the enveloping algebra of $\fr{sl}_2$ is to its minimal primitive quotients. More concretely, the Smith algebra $\mathcal{S}(g)$, associated with the polynomial $g\in\F[h]$, is the deformation of $U(\fr{sl}_2)$ generated by $x,y,h$ with definition relations:
\begin{equation*}
    [h,y] = y, \quad [h,x] = -x \quad \text{and}\quad [y,x]=g(h).
\end{equation*}
$\mathcal{S}(g)$ can be realized as GWA $R(\sigma,a)$ with $R=\F[h, \theta]$, $\sigma(h)=h-1$, $\sigma(\theta)=\theta +g(h)$ and $a=\theta$. In case $g(h)=h$ we retrieve $U(\fr{sl}_2)$.

If $\chara(\F)=0$, then there exists $u\in\F[h]$ such that $g(h)=\sigma(u)-u=u(h-1)-u(h)$; in this case, the element $z=xy-u=yx-\sigma(u)$ generates the center of $\mathcal{S}(g)$ as a polynomial algebra. For any $\lambda\in\F$, the quotient algebra $\mathcal{S}(g)/\langle z-\lambda\rangle$ is a GWA of the type given in \cref{Eg:sl2-prim-quot-as-GWA}.
\end{ex}

With $R$ and $\sigma$ as before, we can also construct the skew Laurent polynomial ring $R[y^{\pm 1}; \sigma]$ (see \cite{GW89}), which is just the GWA $R(\sigma,1)$.

\begin{lemma}
\label{lemma:GWA_iso}
For each $\varphi \in \aut{R}$ we have
\[R(\sigma,a) \simeq R(\varphi \, \sigma \,\varphi^{-1},\varphi(a)).\]
Moreover, if $a$ is a unit in $R$ then $R(\sigma,a) \simeq R(\sigma,1)= R[y^{\pm 1}; \sigma]$.
\end{lemma}
\begin{proof}
The first isomorphism is given by extending $\varphi$ to an isomorphism $\varphi:R(\sigma,a)\to R(\varphi \, \sigma \,\varphi^{-1},\varphi(a))$ by setting $\varphi(x)=x$ and $\varphi(y)=y$. If $a$ is a unit, then extend $\id_R$ to an isomorphism $\varphi:R(\sigma,a)\to R(\sigma,1)$ with $\varphi(x)=x$ and $\varphi(y)=ya$. 
\end{proof}


\begin{cor}
Up to isomorphism, any GWA over $\F[h]$ is of the form $\F[h](\sigma,a)$ where either 
\[\sigma(h)=h-1 \quad \text{or} \quad \sigma(h)=\gamma h,\]
for some $\gamma \in \F^\times$.
\end{cor}


Thus, without loss of generality, in case $R=\F[h]$ we can always assume that $\sigma(h)=h-1$ (named the \textit{classical type}) or that $\sigma(h)=\gamma h$ with $\gamma \in \F^\times$ (named the \textit{quantum type}).

\subsection{$R$-free modules}

An $A$-module $V$ is said to be \textit{$R$-free} if $V \simeq \bigoplus_{i\in I}{}_R R$ as left $R$-modules, for some indexing set $I$; in this situation, we say that the rank of $V$ is $|I|$. In case $R$ has IBN (\textit{invariant basis number}), this notion of rank is invariant under $R$-isomorphism. Note that commutative rings as well as (left) noetherian rings have IBN.

When $R$ has IBN we further define the following full subcategories of $A\text{-}\mathsf{Mod}$:
\[
\begin{aligned}
\mathscr{C} &= \{V \in A\text{-}\mathsf{Mod} \mid \mathrm{Res}_R^A\,V \text{ is free}\},\\
\mathscr{C}_n &= \{V \in A\text{-}\mathsf{Mod} \mid \mathrm{Res}_R^A\,V \simeq ({}_R R)^n\},\\
\mathscr{C}^{\mathrm{fg}} &= \{V \in A\text{-}\mathsf{Mod} \mid \mathrm{Res}_R^A\,V \text{ is finitely generated}\}.
\end{aligned}
\]
We note that $\mathscr{C}^{\mathrm{fg}}$ contains all finite rank $R$-free modules $\bigcup_{n\geq 1} \mathscr{C}_n$; it also contains all finitely supported weight modules, see \cref{SS:PID:weight-mods} below.

\begin{lemma}\label{L:hom:simplification}
Let $V, W$ be $A$-modules with $W$ torsionfree as an $R$-module. Let $\seq{v_i}_{i\in I}$ be a generating set for $V$ as an $R$-module. Then
\begin{equation*}
\Hom_A(V, W)=\pb{\varphi\in\Hom_R(V, W) \mid \varphi(xv_i)=x\varphi(v_i),\ \text{for all $i\in I$}}. 
\end{equation*}
\end{lemma}
\begin{proof}
For the nontrivial inclusion, suppose that $\varphi:V\to W$ is an $R$-map such that $\varphi(xv_i)=x\varphi(v_i)$, for all $i\in I$. Then, given $v\in V$, write $v=\sum_{i\in I}r_i v_i$. So
\begin{align*}
\varphi(xv)&=\sum_{i\in I}\varphi(xr_i v_i)=\sum_{i\in I}\varphi(\sigma^{-1}(r_i) xv_i)=
\sum_{i\in I}\sigma^{-1}(r_i) x\varphi(v_i)\\&=\sum_{i\in I} xr_i\varphi(v_i)=x\varphi(v).
\end{align*}
Moreover,
\begin{equation*}
\sigma(a)\varphi(yv)=yx \varphi(yv)=y \varphi(xyv)=y \varphi(av)=ya \varphi(v)=\sigma(a)y\varphi(v).
\end{equation*}
As $\sigma(a)$ is regular and $W$ is torsionfree, it follows that $\varphi(yv)=y\varphi(v)$.
\end{proof}

Recall that $R^\sigma=\pb{r\in R\mid \sigma(r)=r}$ is a (unital) subring of $R$, usually called the \textit{ring of $\sigma$-invariants} in $R$. The following result records a few useful facts about the center of $A$. In case $R$ is commutative, a description of $\cent{A}$ can also be found in \cite[Corollary 2.0.2]{K01}.

\begin{lemma} \label{L:cent}
The following hold for $A=R(\sigma, a)$:
\begin{enumerate}[label=(\alph*)]
\item $\cent{A}\cap R=R^\sigma\cap\cent{R}$.\label{L:cent:a}
\item If the orbit of $a$ under $\sigma$ is infinite, then $\cent{A}=R^\sigma\cap\cent{R}$.\label{L:cent:b}
\item Assuming that $R$ is commutative, if $\sigma$ has order $k\geq 1$ then $\cent{A}$ is a GWA over $R^\sigma$, where the canonical GWA generators are $x^k$ and $y^k$; otherwise, if $\sigma$ has infinite order, then $\cent{A}=R^\sigma$.\label{L:cent:c}
\end{enumerate} 
\end{lemma}
\begin{proof}
Using the $\bb{Z}$-grading of $A$, we see that $\cent{A}$ is also $\bb{Z}$-graded, in the sense that an element is central if and only if each one of its homogeneous components is central. We denote by $A_n$ the homogeneous component of $A$ of degree $n$. 

So suppose that $z\in A_0=R$. Then $z\in\cent{A}$ if and only if $z\in\cent{R}$ and $[x,z]=0=[y, z]$, and the latter holds if and only if $z\in R^\sigma$. This establishes~\cref{L:cent:a}.

Now suppose that $0\neq z\in A_n$ for some $n\neq 0$. By symmetry, we can assume that $n\geq 1$, so that $z=rx^n$, for some $0\neq r\in R$. Then $[x, z]=0$ is tantamount to $r\in R^\sigma$ and in this case $[y, z]=r(\sigma(a)-\sigma^{1-n}(a))x^{n-1}$. Thus, if $z\in\cent{A}$ then $\sigma^n(a)=a$. So in case the orbit of $a$ under $\sigma$ is infinite, we get $\cent{A}\subseteq R$ and \cref{L:cent:b} now follows from \cref{L:cent:a}.

For \cref{L:cent:c}, assume that $R$ is commutative. By \cref{L:cent:a} we have $\cent{A}\cap R=R^\sigma$. Given $r, s\in R$ with $r\neq 0$ and $n\geq 1$, $[rx^n, s]=r(\sigma^{-n}(s)-s)x^n$. By the above, $rx^n$ is central if and only if $r\in R^\sigma$ and $\sigma^n=\id$. The same result holds for $ry^n$. Hence, if $\sigma$ has infinite order then $\cent{A}=R^\sigma$; otherwise, if $\sigma$ has order $k\geq 1$ then $\cent{A}$ is a free $R^\sigma$-module with basis $\pb{1, X^n, Y^n\mid n\in\bb{Z}_{>0}}$, where $X=x^k$ and $Y=y^k$. It is immediate to see that this is a GWA over $R^\sigma$ associated to the identity automorphism and the central element $a\sigma(a)\cdots \sigma^{k-1}(a)$.
\end{proof}

Next, we see that the existence of \textit{interesting} $R$-free $A$-modules imposes strong restrictions on the center of $A$. 

\begin{prop}\label{P:modules-of-finite-len}
Suppose that there exists some nonzero $R$-free $A$-module with finite length. Then the following hold:
\begin{enumerate}[label=(\alph*)]
\item $\cent{A}\cap R$ is a field;
\item $R^{<\infty}=\left(R^\times\cap\cent{R}\right)\cup\{0\}$,
\end{enumerate}
where $R^{<\infty}=\{r\in \cent{R}\mid \sigma^n(r)=ur, \ \text{for some $0\neq n\in\bb{Z}$ and $u\in R^\times$}\}$.
\end{prop}
\begin{proof}
Note first that $\left(R^\times\cap\cent{R}\right)\cup\{0\}\subseteq R^{<\infty}$ always holds. Suppose that $0\neq r\in R^{<\infty}$, say $\sigma^n(r)=ur$ with $r\in \cent{R}$, $n\neq 0$ and $u\in R^\times$. Without loss of generality, we can assume that $n>0$. Set $\theta=r\sigma(r)\cdots\sigma^{n-1}(r)$. Then $\theta\in\cent{R}$ and $\theta\neq 0$, as $R$ is a domain. Moreover, $\sigma(\theta)=u\theta$. We need to show that $r\in R^\times$ and it suffices to show that $\theta\in R^\times$.

Let $V\neq 0$ be an $R$-free $A$-module of finite length. 
We have that $\theta V$ is an $A$-submodule of $V$ since $R\theta V=\theta RV\subseteq \theta V$,
\begin{align*}
 x\cdot \theta V=\sigma^{-1}(\theta)xV=u^{-1}\theta xV=\theta u^{-1} xV\subseteq \theta V,
\end{align*}
and similarly $y\cdot \theta V\subseteq  \theta V$. Thus, $\theta^k V$ is an $A$-submodule of $V$, for all $k\in\bb{Z}_{\geq 0}$.

Suppose that $\theta V\subsetneq V$. As $\theta\neq 0$ and $V$ is $R$-free, hence $R$-torsionfree (because $R$ is a domain), it follows that $\theta^{k+1} V\subsetneq \theta^k V$, for all $k\in\bb{Z}_{\geq 0}$. So we obtain a proper infinite descending chain of $A$-submodules of $V$, showing that $V$ is not an artinian $A$-module and thus doesn't have finite length. The latter contradicts the hypothesis, so $\theta V=V$. Using the freeness of $V$ over $R$, we conclude that $\theta\in R^\times$, so $r\in R^\times\cap\cent{R}$.

It follows in particular that $R^\sigma\cap\cent{R}\subseteq \left(R^\times\cap\cent{R}\right)\cup\{0\}$, so if $0\neq u\in R^\sigma\cap\cent{R}$, it follows that $uu'=1=u'u$ for some $u'\in R$. Moreover, $u'\in\cent{R}$ and  $u\sigma(u')=\sigma(u)\sigma(u')=\sigma(1)=1$; multiplying on the left by $u'$ shows that $u'\in R^\sigma$. As $R^\sigma\cap\cent{R}$ is a commutative subalgebra of $R$, we conclude that $R^\sigma\cap\cent{R}=\cent{A}\cap R$ is a field, by \cref{L:cent}.
\end{proof}

By the above, unless $\cent{A}\cap R$ is a field, there will be no simple $R$-free $A$-modules.

\begin{rmk}\label{Rk:quotient-by-nonunits}
In case $u\in R^\sigma\cap\cent{R}$ is a nonzero nonunit and $V\neq 0$ is an $R$-free $A$-module, the quotient module $\overline V=V/uV$ can still be interesting to look at. In fact, $\overline V$ is naturally a module for $\overline A=A/uA$, which can be realized as a GWA using $\overline A\simeq \overline R(\overline\sigma, \overline a)$, where $\overline\sigma$ is the automorphism of $\overline R=R/uR$ induced by $\sigma$ and $\overline a$ is the class of $a$ in $\overline R$, see \cite[Proposition 2.12]{BO09}. In case the ideal $uR$ is completely prime and $\overline a\neq 0$, the GWA $\overline A$ is a domain.
\end{rmk}

\begin{ex}[{$\F[H]$}-free modules for $\fr{sl}_2$]\label{Eg:sl2:more}
Recall from~\cref{Eg:sl2-as-GWA} that $U(\fr{sl}_2)$ can be seen as the GWA $A=R(\sigma, a)$ with $R=\F[H, C]$, $\sigma(f(H, C))=f(H-1, C)$ and $a=C-H(H+1)$. Assume that $\chara(\F)=0$. Then $\cent{A}=R^\sigma=\F[C]$ is not a field. Take $\lambda\in \F$ and consider the nonunit $u=C-\lambda\in\cent{A}$. Then, as explained in~\cref{Rk:quotient-by-nonunits}, we can consider $\overline A=U(\fr{sl}_2)/\langle C-\lambda \rangle$, which is a minimal primitive quotient of $U(\fr{sl}_2)$ and can be described as a GWA over $\overline R=R/\langle C-\lambda\rangle\simeq \F[H]$, with $\overline\sigma(H)=H-1$ and $\overline a=\lambda-H(H+1)$. Now we have $\cent{\overline A}=\overline R^{\overline\sigma}=\F$ and we can obtain with this procedure finite length (and simple) $\F[H]$-free modules for $\fr{sl}_2$ and more generally for the Smith algebras in~\cref{Eg:Smith-as-GWA}. This explains the otherwise \textit{ad hoc} assumption or verification in the literature that $\F[H]$-free modules for the aforementioned algebras have central character, even if they are not simple (compare~\cite[Lemma 3.8]{FLM24}).
\end{ex}

\subsection{The rank one modules $V_\p$}\label{SS:C1mods}

Suppose that $V$ is an $A$-module which restricts to the left regular $R$-module $V={}_R R$. Up to isomorphism, any module in $\mathscr{C}_1$ has this form. Set $\p:=x\cdot 1$ and $\q:=y \cdot 1$. Then
\begin{equation*}
a=a\cdot 1=(xy)\cdot 1=x\cdot\q=x\cdot(\q\cdot 1)= (x\q)\cdot 1=(\sigma^{-1}(\q)x)\cdot 1=\sigma^{-1}(\q)\p.
\end{equation*}
A similar calculation using $yx=\sigma(a)$ and applying $\sigma^{-1}$ to the result in $R$ shows that $a=\p\sigma^{-1}(\q)$. However, given our assumptions, if $a=v\p$ for some $v\in R$ then $(a-\p v)\p=a\p-\p a=0$, so $\p v=a=v\p$ and moreover $v$ is uniquely determined by $\p$. In short, we can just talk about divisors of $a$, and $\frac{a}{\p}$ is unambiguous if $\p\mid a$ in $R$. Notice also that $\p$ determines the action of $A$ on $V={}_R R$, as we must have
\[ x \cdot r =x \cdot (r\cdot 1) = (xr) \cdot 1 = ((\sigma^{-1}(r)x)\cdot 1= \sigma^{-1}(r)\p,\] 
and similarly $y\cdot r = \sigma(r)\q$, for all $r\in R$.

The above argument can be reversed and it's straightforward to check that the GWA relations are compatible with the above action.

\begin{defi}\label{D:Vp}
Let $\p$ be a divisor of $a$ and set $\q=\sigma\seq{\frac{a}{\p}}$, so that $a=\sigma^{-1}(\q)\p$.
Let $V_\p=R$ as an abelian group, equipped with the following module structure over $A=R(\sigma,a)$:
\[r' \cdot r =r'r, \quad x \cdot r = \sigma^{-1}(r)\p, \quad y\cdot r = \sigma(r)\q,\]  
for all $r, r'\in R$.
\end{defi}

Hereafter, $\p$ will always denote a divisor of $a$, $\q=\q_\p=\sigma\seq{\frac{a}{\p}}$ and $V_\p$ is the $A$-module given in~\cref{D:Vp}.

\begin{prop}\label{P:class-rk-1-general}
Any $A$-module in $\mathscr{C}_1$ is isomorphic to $V_\p$, for some $\p\in R$ that divides $a$. 
\end{prop}

The following notion will be useful.

\begin{defi}\label{D:sigma-ideal}
We say that a (left) ideal $I$ of $R$ is a (left) \textit{$\sigma$-ideal} if $\sigma(I)= I$.  In case $R$ is (left) noetherian, the weaker condition $\sigma(I)\subseteq I$ implies that $I$ is a (left) $\sigma$-ideal (see e.g.~\cite[Lemma 2.11]{BO09}).
\end{defi}

The next result characterizes the submodule structure of $V_\p$ in a few cases of interest. 

\begin{lemma}
\label{L:subm-Vp-units}
Suppose that $R$ is left noetherian and let $\p$ be a divisor of $a$.
If either one of $\p$ or $\q$ is a central unit in $R$, then the $A$-submodules of $V_{\p}$ are precisely the left $\sigma$-ideals of $R$. 
\end{lemma}
\begin{proof}
Suppose that $\p$ is a central unit (the case when $\q$ is a central unit is symmetric). Note that, since $\sigma(a)=\q \sigma(\p)$ and both $\sigma(a)$ and $\sigma(\p)$ are central and nonzero in the domain $R$, we can conclude that $\q$ is also central.

Let $I$ be an $A$-submodule of $V_{\p}$. The condition $R\cdot I\subseteq I$ says that $I$ is a left ideal of $R$. We have, for $i\in I$,
\begin{align*}
x\cdot i\in I\iff \sigma^{-1}(i)\p\in I\iff \p\sigma^{-1}(i)\in I \iff  \sigma^{-1}(i)\in \p^{-1} I=I, 
\end{align*}
so $x\cdot I\subseteq I$ if and only if $\sigma^{-1}(I)\subseteq I$. By the noetherian hypothesis, $\sigma^{-1}(I)= I$, so $I$ is a left $\sigma$-ideal. Conversely, if $I$ is a left ideal with $\sigma(I)= I$, then the above shows that $I$ is stable under the actions of $R$ and $x$. 
We have 
\begin{align*}
y\cdot I= \sigma(I)\q=\q\sigma(I)=\q I\subseteq I,
\end{align*}
so $I$ is an $A$-submodule of $V_{\p}$ and the correspondence is established.
\end{proof}

%

The next task is to determine the behaviour of $R$-free $A$-modules with respect to the bifunctor $\Hom_A(-,-)$, and in particular the isomorphisms between the $V_\p$.

\begin{lemma}\label{L:hom:units}
Let $\p, \p'$ be divisors of $a$ and consider $\Psi:\Hom_A(V_{\p}, V_{\p'})\to R$ given by $\varphi\mapsto\varphi(1)$.
\begin{enumerate}[label=(\alph*)]
\item $\Psi$ is bijective onto $\pb{v\in R\mid \sigma^{-1}(v)\p'=\p v}$.\label{L:hom:units:a}
\item Fix $\varphi\in \Hom_A(V_{\p}, V_{\p'})$. Then $\varphi$ is injective if and only if $\varphi(1)\neq 0$; $\varphi$ is surjective if and only if $\varphi$ is an isomorphism if and only if $\varphi(1)$ is a unit.\label{L:hom:units:b}
\item If $\p\in\cent{R}$ then $\End_A(V_\p)=R^\sigma\cdot \id_{V_\p}$, where $R^\sigma$ acts by right multiplication on $V_\p$.\label{L:hom:units:c}
\end{enumerate}
\end{lemma}
\begin{proof}
As $V_{\p}$ is a free $R$-module generated by $1$, any $\varphi\in \Hom_R(V_{\p}, V_{\p'})$ is uniquely determined by $v=\varphi(1)$. By~\cref{L:hom:simplification}, $\varphi\in \Hom_A(V_{\p}, V_{\p'})$ if and only if $\p v=\varphi(\p)=x\cdot v=\sigma^{-1}(v)\p'$. In this case, $\varphi(r)=rv$, so $\ker\varphi$ is either $V_{\p}$, if $v=0$, or $\pb{0}$, if $v\neq 0$. Similarly, as $\varphi(V_{\p})=Rv$, it follows that $\varphi$ is surjective if and only if $Rv=R$, i.e., there is $u\in R$ with $uv=1$. As before, this implies that $vu=1$. Thus $v$ is a unit and $\varphi$ is an isomorphism.

For the last claim, in case $\p'=\p$ is central, we get $\p\sigma^{-1}(v)=\sigma^{-1}(v)\p=\p v$, which is equivalent to $v\in R^\sigma$.
\end{proof}

\begin{rmk}
Once again, we are lead to the condition that $R^\sigma\cap\cent{R}=\cent{A}\cap R$ should be a field: if there is any $\p\in\cent{R}$ which divides $a$ and such that $V_\p$ is simple, then Schur's lemma, together with \cref{L:hom:units}\cref{L:hom:units:c}, imply that $R^\sigma$ is a division ring, so that $R^\sigma\cap\cent{R}$ is a field.
\end{rmk}

\begin{cor}\label{C:rank1:iso}
Assume that $R$ is commutative and let $\p, \p'$ be divisors of $a$. If $V_\p \simeq V_{\p'}$, then $\p$ and $\p'$ are associates in $R$ (i.e., $\p'=u\p$, for some unit $u\in R$).
\end{cor}

Let's address a partial converse of \cref{C:rank1:iso} in the general case. Note first that even if $R$ is commutative and $\p$, $\p'$ are divisors of $a$ that are associates, it may occur that $V_\p \not\simeq V_{\p'}$. For example, if $\sigma$ fixes each unit, then the isomorphism condition becomes
$\p'=(\sigma^{-1}(u))^{-1}u\p=u^{-1}u\p = \p$, so in this case $V_\p \simeq V_{\p'} \iff \p=\p'$.

However, when $u$ is a unit, $V_\p$ and $V_{u\p}$ are related in the sense below.

\begin{lemma}\label{L:unit-twists}
For every unit $u \in \cent{R}$ there is an automorphism $\tau_u: A \ra A$ defined on the generators by $\tau_u(r)=r$, for $r\in R$, $\tau_u(x)=ux$ and $\tau_u(y)=yu^{-1}$. Then $\tau_u$ provides an auto-equivalence on the category of $A$-modules, where each module $M$ is mapped to $M^{\tau_u}$, which is the same abelian group as $M$ but with a $\tau_u$-twisted action 
\[b \bullet m := \tau_u(b) \cdot m, \quad \text{for each $b\in A=R(\sigma,a)$.}\]
Using this notation, when $u$ is a central unit in $R$, we have \[V_{u\p} \simeq V_\p^{\tau_u}.\]
\end{lemma}
\begin{proof}
As $\eval{\tau_u}{R}=\id_R$, the identity map is an $R$-module isomorphism from $V_{u\p}$ to $V_\p^{\tau_u}$. By \cref{L:hom:simplification}, it remains to check that $\id(x\cdot 1)=x\bullet\id(1)$, which reduces to $u\p=\tau_u(x)\cdot 1=(ux)\cdot 1=u\p$.
\end{proof}

In particular, we note that if $R$ is a unique factorization domain (UFD), up to unit twists $\tau_u$, there are just finitely many $R$-free modules of rank $1$, and they are parametrized by the divisors of $a$ (up to units).

\subsection{Whittaker modules as $R$-free modules}\label{SS:whittaker}

In case $R$ is an $\F$-algebra and $\zeta\in\F^*$, then the universal Whittaker module of type $\zeta$  constructed in~\cite{BO09} for the GWA $A$ is just the $R$-free module $V_{\zeta^{-1}a}$, with $\p=\zeta^{-1}a$ and $\q=\zeta$, showing how $R$-free modules generalize Whittaker modules.
Moreover, we see that as a particular case of \cref{L:subm-Vp-units} we obtain the characterization in \cite[Lemma 3.9]{BO09} of all submodules of universal Whittaker modules (implying the isomorphism and simplicity criteria for Whittaker pairs, see~\cite[Theorem 3.12, Corollary 3.13]{BO09}). We remark however that in~\cite{BO09} the authors are working in the more general setting of iterated GWAs.

\section{The case where $R$ is a UFD}\label{S:UFD}

In this section we assume that $R$ is a commutative unique factorization domain (UFD) and we continue to assume that $A=R(\sigma, a)$ with $a\neq 0$. Unsurprisingly, our sharpest results are in fact obtained when $R$ is a PID.

The dynamics of the action of the automorphism $\sigma$ on the set of units, irreducible elements and, more generally, on ideals of $R$ will be of utmost relevance to study the structure of the $R$-free modules, so we fix the notation $\mathcal G=\langle \sigma \rangle$ for the subgroup of $\aut{R}$ generated by $\sigma$, the defining automorphism of $A$.

\subsection{Irreducibles, associates and $\mathcal G$-orbits}

We write $r\sim s$ if $r$ and $s$ are associates in $R$ and $[r]$ for the equivalence class of $r$ with respect to this relation. So $[0]=\{0\}$ and $[1]=R^\times$, the group of units of $R$. Then $R/{\sim}$ parametrizes the set of principal ideals in $R$. Similarly, let $\Irr(R)=\{[r]\in R /{\sim} \mid  \text{$r$ is irreducible}\}$ be the set of irreducibles in $R$, up to associates. Then, in case $R$ is further assumed to be noetherian, $\Irr(R)$ parametrizes the set of prime ideals of height $1$ in $R$, by Kaplansky's theorem.

For an element $r\in R$, we write $\Irr (r)$ for the multiset of equivalence classes of irreducible factors of $r$. In other words, 
if $r=ur_1r_2 \cdots r_n$ with $r_i$ irreducible and $u$ a unit, then $\Irr (r)=\{[r_1], \ldots, [r_n]\}$, where the 
$[r_i]$ are not necessarily distinct. We write $\degg(r)=|\Irr(r)|$ for the number of such factors, counting multiplicities.

Since the action of $\mathcal G$ on $R$ preserves associates and irreducibility, it follows that $\mathcal G$ acts on $R /{\sim}$ and on $\Irr(R)$. We use $\orb$ to denote the orbits in either one of these contexts. In particular,
\[\orb([r]) = \{[\sigma^{k}(r)] \mid  k\in \bb{Z}\}.\] 
To avoid notational complexity, we will often just write $\orb(r)$ in place of $\orb([r])$ whenever it is clear that we're working modulo associates. The space of $\mathcal G$-orbits in $R /{\sim}$ can be identified with $R/{\sim}_\sigma$, where 
\[r \sim_{\sigma} s \iff \sigma^k(r)\sim s, \text{ for some $k\in\bb{Z}$} \iff \orb([r]) =\orb([s]).\] 

\subsection{Cyclic submodules of $V_\p$}

We note that, since $V_{\p} = {}_RR$ as an $R$-module, any submodule of $V_{\p}$ is an ideal of $R$. The next result uses only the fact that $R$ is a commutative domain.

\begin{lemma}\label{L:radI-subm}
Let $\p$ be a divisor of $a$. 
\begin{enumerate}[label=(\alph*)]
\item \label{L:radI-subm:c} The module $V_{\p}$ is indecomposable.
\item \label{L:radI-subm:d} If $V_{\p}$ has finite length then it has a unique simple submodule; in other words, its socle $\soc(V_\p)$ is simple.
\item \label{L:radI-subm:b} If $I$ is a proper $A$-submodule of $V_\p$, then so is
\begin{align*}
\sqrt{I}=\{r\in R\mid r^n\in I, \text{ for some $n\geq 1$}\}. 
\end{align*}
\end{enumerate}
\end{lemma}
\begin{proof}
Let $I$ and $J$ be $R$-submodules of $V_\p$ (thus ideals of $R$) and suppose that $I\cap J=0$. Then $IJ\subseteq I\cap J=0$, so either $I=0$ of $J=0$, because $R$ is a domain. It follows that $V_\p$ contains no nontrivial direct sums of $R$-submodules. Then \cref{L:radI-subm:c} follows because $V_\p$ is already indecomposable as an $R$-module. Similarly, since by hypothesis $\soc(V_\p)\neq 0$ and $\soc(V_\p)$ is a direct sum of simple $A$-submodules of $V_\p$, the previous observations show that $\soc(V_\p)$ must be simple, hence the unique simple $A$-submodule of $V_\p$, establishing \cref{L:radI-subm:d}.

As $\sqrt{I}$ is a proper ideal of $R$, it remains to show that it is closed under the actions of $x$ and $y$. Suppose that $r\in\sqrt{I}$, say $r^n\in I$, for some $n\geq 1$. Then 
\begin{align*}
(x\cdot r)^n=(\sigma^{-1}(r))^n\p^n=((\sigma^{-1}(r^n))\p)\p^{n-1}=(x\cdot r^n)\p^{n-1}\in I, 
\end{align*}
so $x\cdot r\in \sqrt{I}$. Similarly, $y\cdot r\in \sqrt{I}$ and $\sqrt{I}$ is an $A$-submodule of $V_\p$.
\end{proof}

\begin{prop}
\label{lemma:sub_condition}
Suppose that $R$ is a UFD and let $\p$ be a divisor of $a$. Recall that $\q=\sigma(a/\p)$.
Then the principal $R$-ideal $\langle g \rangle$ is a submodule of $V_\p$ if and only if
\begin{align}\label{E:irrg-p}
\Irr (g) \subseteq \sigma^{-1}(\Irr (g)) \cup \Irr(\p) 
\end{align}
and
\begin{align}\label{E:irrg-q}
\Irr (g) \subseteq \sigma(\Irr (g)) \cup \Irr(\q), 
\end{align}
where union and inclusion are operations on multisets, taking multiplicity into account. In such a case, $\langle g \rangle\simeq V_{\p'}$, where $\p'=\frac{\sigma^{-1}(g)\p}{g}$ and $\q'=\q_{\p'}=\frac{\sigma(g)\q}{g}$.
\end{prop}
\begin{proof}
Let $I=\langle g\rangle=Rg$. 
Then $I$ is an $A$-submodule of $V_\p$ if and only if $x \cdot s \in I$ and $y\cdot s \in I$ for all $s \in I$. Since $I$ is principal we have $s=rg$, and using the definition of the $V_\p$-action these conditions become
\[\sigma^{-1}(r)\sigma^{-1}(g)\p \in I \quad\text{ and }\quad \sigma(r)\sigma(g)\q\in I.\]
Clearly, this holds for all $r\in R$ if and only if it holds for $r=1$, so we get 
\[g|\sigma^{-1}(g)\p \text{ and } g|\sigma(g)\q.\]
The condition on the left says that every irreducible factor of $g$ should be either a factor of $\sigma^{-1}(g)$ or of $\p$, counting multiplicities, while the condition on the right says that every irreducible factor of $g$ is a factor of either $\sigma(g)$ or of $\q$.

Now suppose that $\langle g \rangle$ is a submodule of $V_\p$. As $R$ is a domain, $\langle g \rangle\simeq {}_RR$ as left $R$-modules so, by \cref{P:class-rk-1-general}, $\langle g \rangle\simeq V_{\p'}$ for some $\p'$ which divides $a$. By~\cref{E:irrg-p}, $\sigma^{-1}(g)\p g^{-1}\in R$ and
\begin{align*}
x\cdot g=\sigma^{-1}(g)\p=(\sigma^{-1}(g)\p g^{-1})g.
\end{align*}
So we conclude that $V_{\p'} \simeq \langle g \rangle \subseteq V_\p$ where $\p'=\frac{\sigma^{-1}(g)\p}{g}$, as claimed, and $\q'=\sigma(a/\p')=\frac{\sigma(g)\q}{g}\in R$, by~\cref{E:irrg-q}. 
\end{proof}



Our next result shows that we can analyze each $\mathcal G$-orbit separately in testing for submodules of $V_\p$.

\begin{lemma}
\label{lemma:orbit_partition}
Suppose that $R$ is a UFD and let $\p$ be a divisor of $a$.
Let $0\neq g\in R\setminus R^\times$. Factor $g$ into irreducibles and group the factors according to their $\mathcal G$-orbits; in other words, write \[g=g_1g_2\cdots g_n\] such that, for $x,y \in \Irr(g)$, we have 
\[x \sim_{\sigma} y \Leftrightarrow x,y\in \Irr(g_i), \text{ for a unique $i$}.\]

Then $\langle g\rangle$ is a submodule of $V_\p$ if and only if $\langle g_i\rangle$ is a submodule of $V_\p$ for each $1 \leq i \leq n$. In this case we have $\langle g\rangle=\cap_{i=1}^n\langle g_i\rangle$. 
\end{lemma}
\begin{proof}
Partition the factors $\p$ and $\q$ according to the same $\mathcal G$-orbits as $g_1, \ldots g_n$:
\[\p=\p_0\p_1\p_2\cdots \p_n \text{ and } \q=\q_0\q_1\q_2 \ldots \q_n,\]
where the additional factors $\p_0$ and $\q_0$ are the product of the irreducibles that do not belong to any of the orbits of the factors of $g$.

\cref{lemma:sub_condition} shows that $\langle g \rangle$ is a submodule of $V_\p$ if and only if
\[\Irr (g) \subseteq \sigma^{-1}(\Irr (g)) \cup \Irr(\p) \quad\text{ and } \quad\Irr (g) \subseteq \sigma(\Irr (g)) \cup \Irr(\q).\]
But considering the $\mathcal G$-orbit of an irreducible $z$, we see that if $z \in \Irr(g_i)$, then $z\in \sigma^{-1}(\Irr (g))$ if and only if $z\in \sigma^{-1}(\Irr (g_i))$, while $z \in \Irr(\p)$ if and only if $z \in \Irr(\p_i)$. Together with the analogous argument for \cref{E:irrg-q}, we conclude that $\langle g \rangle$ is a submodule of $V_\p$ if and only if, for each $1\leq i \leq n$, we have
\begin{align}\label{E:lemma:orbit_partition}
\Irr (g_i) \subseteq \sigma^{-1}(\Irr (g_i)) \cup \Irr(\p_i) \quad\text{ and } \quad\Irr (g_i) \subseteq \sigma(\Irr (g_i)) \cup \Irr(\q_i). 
\end{align}
As $\Irr(\p_i)\subseteq \Irr(\p)$ and $\Irr(\q_i)\subseteq \Irr(\q)$, we conclude that~\cref{E:lemma:orbit_partition} is the condition for $\langle g_i \rangle$ to be a submodule of $V_\p$. 

The intersection property $\langle g\rangle=\cap_{i=1}^n\langle g_i\rangle$ follows because the $g_i$ are pairwise coprime, since their irreducible factors belong to different $\mathcal G$-orbits.
\end{proof}


Let $[z] \in \Irr(R)$ and fix the $\mathcal G$-orbit $\orb([z])$. For any $0\neq f\in R$, write $f=f_0f_{z}$, where $\Irr(f_{z})\subseteq \orb([z])$ and $\Irr(f_0)\cap \orb([z])=\varnothing$ (of course, $f_0$ and $f_{z}$ are only determined up to units). Thus, up to units, the irreducible factors of $f_z$ all have form $\sigma^{k}(z)$, with $k\in\bb Z$. Finally, we define a function $\olfz:\bb{Z} \ra \bb{Z}$ that keeps track of the multiplicities of the irreducible factors of $f_{z}$, such that $\olfz(k)$ is the maximum of all $n\geq 0$ such that $(\sigma^k(z))^n$ divides $f_{z}$. Note that $\olfz=\olfz[f_z]$. 

To avoid cumbersome notation, whenever $[z]\in \Irr(R)$ is fixed and understood from the context, we use just $\ol{f}$ in place of $\olfz$.

By \cref{lemma:orbit_partition}, when studying submodules of $V_\p$ of the form $\langle g\rangle$, we can focus on a single $\mathcal G$-orbit. So suppose that $0\neq g\sim g_z$.
If $\orb([z])$ is infinite, then $\ol{g}$ has finite support (i.e., there are only finitely many $k\in\bb Z$ such that $\ol{g}(k)\neq 0$) and
\[g \sim \prod_{k\in \bb{Z}}\big(\sigma^{k}(z)\big)^{\ol{g}(k)}.\]
What is more, the assignment $g \mapsto \ol{g}$ is a bijection between elements $0\neq g$ such that $g\sim g_z$, and nonnegative functions $\bb{Z} \ra \bb{Z}$ with finite support.

On the other hand, if $|\orb([z])|=n$ is finite, we have
\[g \sim \prod_{k=0}^{n-1}\big(\sigma^{k}(z)\big)^{\ol{g}(k)}\]
and then the assignment $g \mapsto \ol{g}$ is a bijection between elements $0\neq g$ such that $g\sim g_z$, and nonnegative $n$-periodic functions $\bb{Z} \ra \bb{Z}$, which we can think of as functions $\bb{Z}_n \ra \bb{Z}$.

Here, $\ol{\sigma(g)}(k)=\ol{g}(k-1)$ so applying $\sigma$ to $g$ corresponds to translating the graph of $\ol{g}$ one step to the right; similarly $\ol{\sigma^{-1}(g)}(k)=\ol{g}(k+1)$. 

We now derive a condition for $\langle g \rangle$ to be a submodule of $V_\p$ in terms of the functions $\ol{g}$, $\ol{\p}$ and $\ol{\q}$ and the finite difference operators $\Delta, \nabla : \bb{Z}^{\bb{Z}} \ra \bb{Z}^{\bb{Z}}$ given by
\begin{align}\label{E:finite-diff-opers}
(\Delta f)(k) = f(k+1)-f(k) \quad \text{ and } \quad (\nabla f)(k) = f(k)-f(k-1). 
\end{align}

\begin{lemma}
\label{lemma:sub_condition_new}
Fix $[z] \in \Irr(R)$ and assume that $0\neq g\sim g_z$. Then $\langle g \rangle$ is a submodule of $V_\p$ if and only if
\[\ol{\p} \geq -\Delta \ol{g}  \quad \text{ and } \quad \ol{\q} \geq \nabla \ol{g}.\]
\end{lemma}
\begin{proof}
In the present context, the first condition in~\cref{E:lemma:orbit_partition} for $\langle g \rangle$ to be a submodule of $V_\p$ says that, for each irreducible factor $\sigma^k(z)$ of $g$ (counting multiplicities), either $\sigma^k(z) \in \Irr(\sigma^{-1}(g))$ or $\sigma^k(z) \in \Irr(\p_z)$. Since $\ol{g}(k)$ is the multiplicity of $\sigma^k(z)$ in $g$, this condition is equivalent to 
\[\ol{\p}(k)=\ol{\p_z}(k) \geq \ol{g}(k)-\ol{\sigma^{-1}(g)}(k) = \ol{g}(k)-\ol{g}(k+1) = -(\Delta \ol{g})(k),\] for each $k \in \bb{Z}$, which means that $\ol{\p} \geq  -\Delta \ol{g}$, as stated.

Similarly, the second condition in~\cref{E:lemma:orbit_partition} can be restated as \[\ol{\q}(k)=\ol{\q_z}(k) \geq \ol{g}(k)-\ol{\sigma(g)}(k)=\ol{g}(k)-\ol{g}(k-1)=(\nabla \ol{g})(k),\]
for each $k \in \bb{Z}$, which means that $\ol{\q} \geq  \nabla \ol{g}$, as stated.
\end{proof}

\begin{ex}
\label{ex:graphs}
To illustrate the preceding considerations, suppose that $[z] \in \Irr(R)$ with $\orb([z])$ infinite and $0\neq g\sim g_z$. Let $\ol{g}: \bb{Z} \ra \bb{Z}$ be the function depicted in green below:

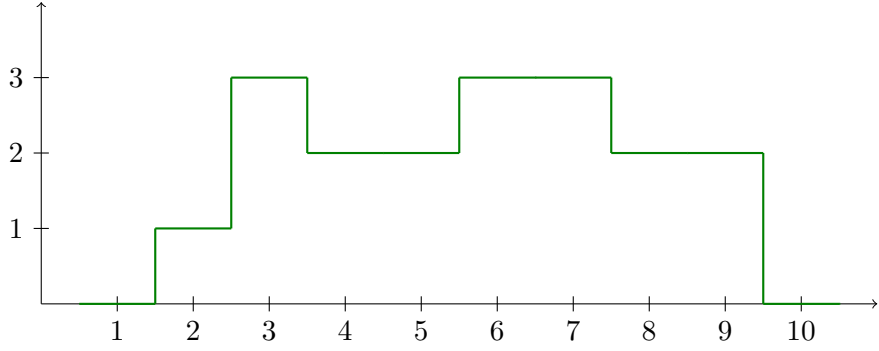
\[\begin{tikzpicture}
    \draw[->] (0,0) -- (11,0) node[below] {};
    \draw[->] (0,0) -- (0,4) node[left] {};
    \foreach \x in {1,2,...,10}
        \draw (\x,0.1) -- (\x,-0.1) node[below] {\x};
    \foreach \y in {1,2,...,3}
        \draw (0.1,\y) -- (-0.1,\y) node[left] {\y};
        
    \plot{{0,1,3,2,2,3,3,2,2,0}}{green!50!black}{thick}
\end{tikzpicture}\]

We compute and plot the discrete derivatives, ignoring the negative values. In the diagram below, 
we have $\max\{0,-\Delta(\ol{g})\}$ in red, and $\max\{0,\nabla(\ol{g})\}$ in blue:
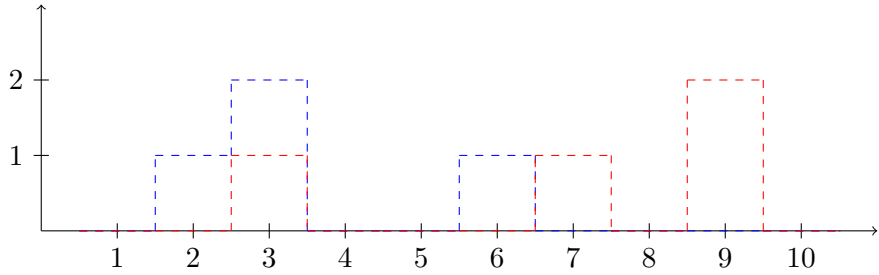
\[\begin{tikzpicture}
    \draw[->] (0,0) -- (11,0) node[below] {};
    \draw[->] (0,0) -- (0,3) node[left] {};
    \foreach \x in {1,2,...,10}
        \draw (\x,0.1) -- (\x,-0.1) node[below] {\x};
    \foreach \y in {1,2}
        \draw (0.1,\y) -- (-0.1,\y) node[left] {\y};
        
    \plot{{0,1,2,0,0,1,0,0,0,0}}{blue}{dashed}
    \plot{{0,0,1,0,0,0,1,0,2,0}}{red}{dashed}
\end{tikzpicture}\]
Here, according to \cref{lemma:sub_condition_new}, in order for $\langle g \rangle$ to be a submodule of $V_\p$, we see that $\p$ must be divisible by $\sigma^{3}(z)\sigma^{7}(z)(\sigma^{9}(z))^2$, while $\q$ must be divisible by $\sigma^{2}(z)(\sigma^{3}(z))^2\sigma^{6}(z)$. In other words, each increase in the function $\ol{g}$ corresponds to factors of $\q$ and each decrease in $\ol{g}$ corresponds to factors of $\p$.
\end{ex}

\begin{rmk}
We note that the only difference between the finite and infinite orbit cases is when $0\neq g\sim g_z$ is constant such that $\nabla(\ol{g})=0=\Delta(\ol{g})$. If $\orb([z])$ is infinite, this implies $\ol{g}=0$ since $\ol{g}$ must have finite support; thus $g$ is a unit and $\langle g\rangle=V_\p$. But if $|\orb([z])|=n<\infty$, it is possible for $\ol{g}$ to be constant and equal to $k\in\bb{Z}_{>0}$, which gives $g \sim (z \sigma(z) \sigma^2(z) \cdots \sigma^{n-1}(z))^k$. In the latter case, since $0\neq z\in R^{<\infty}\setminus R^\times$, \cref{P:modules-of-finite-len} implies that all free $A$-modules have infinite length. 
\end{rmk}

To describe the submodule structure of each $V_\p$ we should solve the opposite problem: given $[z]\in\Irr(R)$ and fixed functions $\ol{\p}$ and $\ol{\q}$, what choices of $g$ are possible in order for $\langle g \rangle$ to be a submodule of $V_\p$. The following notion will be helpful.

\begin{defi}
Let $r, s\in\Irr(R)$ and assume that there is some $n\geq 0$ such that $s\sim \sigma^n(r)$, with $n$ minimum with this property. 
We define the corresponding \textbf{chain product} from $[r]$ to $[s]$ as
\[\ca{P}(r,s) = \prod_{i=0}^n \sigma^i(r)
.\] 
\end{defi}



\begin{lemma}\label{L:basic-submodules}
Suppose that $\q_0\in \Irr(\q)$ and that $\sigma^{n}(\q_0)=\p_0 \in \Irr(\p)$, for some $n\geq 0$. 
In case the orbit of $[\q_0]$ is finite, we shall also require that $n \leq |\orb([\q_0])|-1$ (so $n\geq0$ is minimal such that $\p_0\sim \sigma^{n}(\q_0)$).
For such $\p_0, \q_0$, the ideal 
\[\langle \ca{P}(\q_0,\p_0) \rangle = \langle \prod_{i=0}^n \sigma^i(\q_0) \rangle\]
generated by the chain product from $[\q_0]$ to $[\p_0]$ is a submodule of $V_\p$ isomorphic to $V_{\p'}$, where $\p'= \sigma^{-1}(\q_0)\frac{\p}{\p_0}$ and $\q'=\sigma(\p_0)\frac{\q}{\q_0}$.
\end{lemma}
\begin{proof}
This follows directly from \cref{lemma:sub_condition}, as every factor of the chain product $\ca{P}(\q_0,\p_0)=\prod_{i=0}^n \sigma^i(\q_0)$ is a factor of $\sigma(\ca{P}(\q_0,\p_0))$, except the first factor $\q_0$ which divides $\q$, and every factor of $\ca{P}(\q_0,\p_0)$ is a factor of $\sigma^{-1}(\ca{P}(\q_0,\p_0))$, except the last factor $\sigma^n(\q_0)= \p_0$, which is a factor of $\p$. The formulas for $\p'$ and $\q'$ also follow from \cref{lemma:sub_condition}.
\end{proof}

Henceforth, we will use {\bf basic chain products} to refer to those chain products $\ca{P}(\q_0,\p_0)$ as in
\cref{L:basic-submodules} such that, in case the orbit of $[\q_0]$ is finite then $n < |\orb([\q_0])|-1$ (i.e., so that $\{[\sigma^i(\q_0)]\mid 0\leq i\leq n\}$ is not the whole orbit), and {\bf basic submodules} the submodules $\langle \ca{P}(\q_0,\p_0) \rangle$ of $V_\p$ that they generate.


Note that on taking a basic submodule $V_{\p'}$ of $V_{\p}$, the pair $(\p',\q')$ is essentially obtained from $(\p,\q)$ by switching two irreducible factors $\q_0\mid \q$ and $\p_0\mid \p$ (and applying $\sigma^{-1}$ and $\sigma$, respectively).

\begin{thm}\label{T:all-principal-submodules}
The $A$-submodules of $V_\p$ which are cyclic over $R$ are those of the form $\langle g\rangle$, where $g$ is either $0$ or a product of basic chain products and products of full finite orbits:
\begin{align}\label{E:all-principal-submodules:g}
g \sim \big(\prod_{i=1}^n \ca{P}(\q_i,\p_i)\big) \cdot \big(\prod_{j=1}^m \prod_{k=1}^{|\orb(z_j)|} \sigma^{k}(z_j)\big), 
\end{align}
for $\p_i, \q_i\in\Irr(R)$ such that $\prod_{i=1}^n \q_i\mid \q$ and $\prod_{i=1}^n \p_i\mid \p$, and where the $z_j$ are irreducible elements of some (not necessarily distinct) finite orbits.

Moreover, we may always choose the basic chains so that, for each pair of indices $i,j$, we have $\sigma(\p_i) \not\sim \q_j$ (otherwise concatenate the chains), and such that $\ca{P}(\q_i,\p_i)$ and $\ca{P}(\q_j,\p_j)$ are either coprime or one divides the other. 
\end{thm}

Note that, with $g$ as in \cref{E:all-principal-submodules:g} and using \cref{lemma:sub_condition}, we have that $\langle g\rangle\simeq V_{u\p'}\simeq V_{\p'}^{\tau_u}$ (see~\cref{L:unit-twists}), where $\p'=\p\prod_{i=1}^n \frac{\sigma^{-1}(\q_i)}{\p_i}$ and $u\in R^\times$ is a unit which can arise from the full products of finite orbits $\prod_{k=1}^{\ell} \sigma^{k}(z)$ with $|\orb(z)|=\ell$. This is because such an orbit will induce the unit factor $\frac{z}{\sigma^\ell(z)}$ in the formula from \cref{lemma:sub_condition}, since $\sigma^\ell(z)$ is only assumed to be an associate of $z$. So, in particular, up to the unit twists defined in~\cref{L:unit-twists}, the isomorphism class of $\langle g\rangle$ doesn't depend on the factors which are products of finite orbits.

\begin{proof}
Assume that $\langle g \rangle$ is a submodule of $V_\p$. Without loss of generality, we may assume that $g$ is neither zero nor a unit and that all irreducible factors of $g$ belong to a single orbit $\orb([z])$, for some $z\in\Irr(R)$, by \cref{lemma:orbit_partition}. 
Identify $g$ with the corresponding function $\ol{g}:\bb{Z} \ra \bb{Z}$, as before. 

Assume first that $\orb([z])$ is infinite. Then $\ol{g}$ is nonnegative and has finite support, so it can be expressed uniquely as a sum of indicator functions of its superlevel sets, as
\[\ol{g} = \sum_{k=1}^m 1_{S_k},\quad  \text{ where } S_k=\{ j \in \bb{Z} \mid  \ol{g}(j)) \geq k \}. \]
So $\ol{g}$ can be expressed uniquely as a sum of indicator functions on $\bb{Z}$-intervals
\[\ol{g} = \sum_{i=1}^n 1_{[a_i,b_i]},\]
where $n$ is as small as possible (so e.g.\ $1_{[a,b]}+1_{[b+1,c]}$ is replaced by $1_{[a,c]}$) and such that for each $i,j$, either $[a_i,b_i]$ and $[a_j,b_j]$ are disjoint or one includes the other (so e.g.\ $1_{[a_i,b_i]}+1_{[a_j,b_j]}$ is replaced by $1_{[a_i,b_j]}+1_{[a_j,b_i]}$ if $a_i < a_j \leq b_i < b_j$).
Now let $\q_i:=\sigma^{a_i}(z)$ and $\p_i:=\sigma^{b_i}(z)=\sigma^{b_i-a_i}(\q_i)$, so that $g\sim\prod_{i=1}^n \ca{P}(\q_i,\p_i)$.

It remains to show that $\prod_{i=1}^n \ca{P}(\q_i,\p_i)$ generates a submodule of $V_\p$ if and only if $\prod_{i=1}^n \q_i\mid \q$ and $\prod_{i=1}^n \p_i\mid \p$. Indeed, if $\ol{g} = \sum_{i=1}^n 1_{[a_i,b_i]}$, then 
\[(-\Delta \ol{g})(k)=\ol{g}(k)-\ol{g}(k+1) =\sum_{i=1}^n \delta_{k,b_i} - \delta_{k,a_i-1},\] where $\delta$ is the Kronecker delta function. Since, for all $i,j$, $a_i-1 \neq b_j$, we have $\max\{0,-\Delta\ol{g}\}=\sum_{i=1}^n 1_{\{b_i\}}$ and the first condition of 
\cref{lemma:sub_condition_new} for $\langle g\rangle$ to be a submodule of $V_\p$ becomes
$\ol{\p}\geq  \sum_{i=1}^n 1_{\{b_i\}}$, which is equivalent to $\prod \p_i \mid \p$.

Analogously, we have
$(\nabla \ol{g})(k)=\ol{g}(k)-\ol{g}(k-1) = \sum_{i=1}^n \delta_{k,a_i} - \delta_{k,b_i+1}$, so the second condition from
\cref{lemma:sub_condition_new} becomes $\ol{\q}\geq  \sum_{i=1}^n 1_{\{a_i\}}$, which is equivalent to $\prod \q_i \mid \q$.

The same argument works when $|\orb([z])|=\ell<\infty$, except that the function $\ol{g}: \bb{Z}_\ell \ra \bb{Z}$ is expressed uniquely as
\[\ol{g} = c\cdot \id + \sum_{i=1}^n 1_{[a_i,b_i]}, \]
where $c=\min\ol{g}\geq 0$, $n$ is as small as possible, the intervals $[a_i, b_i], [a_j, b_j] \subseteq \bb{Z}_\ell$ are taken circularly (modulo $\ell$) and assumed to be either disjoint or one contained in the other.
The term $c\cdot \id$ corresponds to a factor $\big(\prod_{i=1}^\ell\sigma^i(z)\big)^c$, a power of the full orbit product of $[z]$.
\end{proof}


\begin{ex}
Let us look again at the generator $g$ of the ideal from \cref{ex:graphs}, which factors as a product of four chains:
\[g = \ca{P}(\sigma^2(z),\sigma^9(z)) \cdot \ca{P}(\sigma^3(z),\sigma^9(z)) \cdot \ca{P}(\sigma^3(z),\sigma^3(z)) \cdot \ca{P}(\sigma^6(z),\sigma^7(z)),\]
as visualized below.
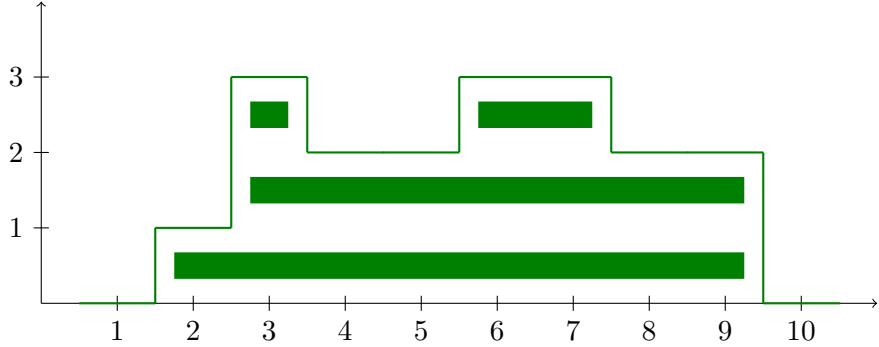
\[\begin{tikzpicture}
    \draw[->] (0,0) -- (11,0) node[below] {};
    \draw[->] (0,0) -- (0,4) node[left] {};
    \foreach \x in {1,2,...,10}
        \draw (\x,0.1) -- (\x,-0.1) node[below] {\x};
    \foreach \y in {1,2,...,3}
        \draw (0.1,\y) -- (-0.1,\y) node[left] {\y};
        
    \plot{{0,1,3,2,2,3,3,2,2,0}}{green!50!black}{thick}
    
    \draw[line width=10pt, green!50!black] (2-0.25,0.5)--(9+0.25,0.5);
    \draw[line width=10pt, green!50!black] (3-0.25,1.5)--(9+0.25,1.5);
    \draw[line width=10pt, green!50!black] (3-0.25,2.5)--(3+0.25,2.5);
    \draw[line width=10pt, green!50!black] (6-0.25,2.5)--(7+0.25,2.5);
\end{tikzpicture}\]
As we determined in the example, the product of left end-points of the chains must divide $\q$ and the product of right endpoints must divide $\p$ in order for $\langle g \rangle$ to be a submodule of $V_\p$.
\end{ex}

\section{The case where $R$ is a PID}\label{S:PID}

In this section, assume that $R$ is a principal ideal domain (PID). All of the previous results apply, but in the PID case we shall be able to give full details concerning the simplicity, submodule structure and composition series of the modules $V_\p$.

\subsection{Simplicity and submodule structure of $V_\p$}\label{SS:PID:simp-subm}

Since $R$ is a PID, all submodules of $V_\p$ are $R$-cyclic, so \cref{T:all-principal-submodules} provides a complete description of all submodules of $V_\p$. In addition, given nonzero submodules $\langle g \rangle$ and $\langle g' \rangle$ with $g, g'$ decomposed as in  \cref{E:all-principal-submodules:g}, then 
 $\langle g \rangle \subseteq \langle g' \rangle$ if and only if $g'\mid g$, so that $g$ is obtained from $g'$ through multiplication by a number of basic chain products and a number of full orbit products. So indeed \cref{T:all-principal-submodules} describes the submodule structure of $V_\p$ completely.

We define a relation $\lleq$ on $\Irr(R)$ by setting
\[[r]\lleq [s] \iff [\sigma^n(r)]=[s] \quad \text{for some $n\geq 0$.} \]
Then $[r]$ and $[s]$ are comparable if and only if they belong to the same $\mathcal G$-orbit. For convenience, we may also write $r \lleq s$ in place of $[r]\lleq [s]$. We will write $r \llneq s$ to mean that $r \lleq s$ but $r\not\sim s$. Although defined in general, this notion is interesting only for infinite $\mathcal G$-orbits.

Suppose that $V_\p$ is simple, for some $\p\mid a$. Then in particular $V_\p$ has finite length and thus \cref{P:modules-of-finite-len} implies that $\orb([z])$ is infinite for all $z\in\Irr(R)$. This observation together with \cref{T:all-principal-submodules} give the following simplicity criterion.  

\begin{cor}
\label{cor:simplicity}
Assume that $R$ is a PID. The module $V_\p$ is simple if and only if the following conditions hold:
\begin{enumerate}[label=(\arabic*)]
\item all $\mathcal G$-orbits in $\Irr(R)$ are infinite, and
\item if there are irreducibles $\p_0\mid \p$ and $\q_0\mid \q$ in the same $\mathcal G$-orbit, then $\p_0\llneq\q_0$; in other words, we have: \[\sigma^{n}(\q_0) \in \Irr(\p) \text{ for some } \q_0 \in \Irr(\q) \implies n<0.\]
\end{enumerate}
\end{cor}

Now we can explicitly describe all maximal submodules of $V_\p$.

\begin{cor}
\label{cor:max_subs}
Assume that $R$ is a PID. The maximal submodules of $V_\p$ come in three types:
\begin{enumerate}[label=(\Roman*)]
\item $\langle 0\rangle$, if $V_\p$ is simple (see \cref{cor:simplicity}).
\item \label{cor:max_subs:fin:not-p-and-q} Full finite-orbit maximal submodules: 
\[\left \langle \prod_{i=1}^{\ell}\sigma^i(z)\right\rangle,\]
where $z\in\Irr(R)$ with $|\orb(z)|=\ell<\infty$ and either:
\begin{enumerate}[label=(\roman*)]
\item $\orb(z)\cap \Irr(\p) = \varnothing$;
\item $\orb(z)\cap \Irr(\q) = \varnothing$;
\item $\orb(z)\cap \Irr(\p) = \{\p_0\}$, $\orb(z)\cap \Irr(\q) = \{\q_0\}$ and $\p_0=\sigma^{\ell-1}(\q_0)$.
\end{enumerate}
\item \label{cor:max_subs:fin:basic} Maximal basic submodules:
\[\langle \ca{P}(\q_0,\p_0) \rangle = \left\langle \prod_{i=0}^n \sigma^i(\q_0)\right\rangle,\]
where $\q_0 \in \Irr(\q)$, $\sigma^{n}(\q_0)=\p_0 \in \Irr(\p)$ for some $n\geq 0$, such that there are no $0\leq i\leq j\leq n$ with $j-i<n$, $\sigma^i(\q_0) \in \Irr(\q)$ and $\sigma^j(\q_0) \in \Irr(\p)$.
\end{enumerate}
\end{cor}

Our next result is somewhat surprising: it characterizes the GWAs over a PID for which all $R$-free modules of rank $1$ are simple.

\begin{thm} 
\label{thm:all_simple}
Let $A=R(\sigma,a)$ be a generalized Weyl algebra over the principal ideal domain $R$. Then all $R$-free modules of rank $1$ are simple if and only if either $R$ is a field, or $A$ is a simple ring.
\end{thm}
\begin{proof}
Assume that $V_\p$ is simple, for all $\p\mid a$. So $A$ has at least one nonzero $R$-free module of finite length (e.g.\ $V_1$ is simple). If $\sigma^n=\id_R$ for some $n\neq 0$, then $R=R^{<\infty}$ is a field, by \cref{P:modules-of-finite-len}, so we can assume that $\sigma$ has infinite order. Suppose that $I$ is a $\sigma$-ideal of $R$. Then (see \cref{L:subm-Vp-units}) $I$ is a submodule of $V_1$, so either $I=\langle 0\rangle$ or $I=R$. Lastly, let $n \geq 1$ and set $g=\gcd\{a, \sigma^n(a)\}$. As $R$ is a PID, we have $Ra + R\sigma^n(a)=\langle g\rangle$. If $g\notin R^\times$ then we can choose some $\p\in\Irr(g)$, and by \cref{P:modules-of-finite-len} we know that $\orb(\p)$ is infinite. As $\p\mid a$, we can consider the module $V_\p$, which by hypothesis is simple, and we have $a=\p\sigma^{-1}(\q)$. We also have $\p\mid \sigma^n(a)$, so $\sigma^{-n}(\p)\mid a$ and $\sigma^{-n}(\p)\not\sim \p$ because $n\neq 0$. So $\sigma^{-n}(\p)\mid \sigma^{-1}(\q)$, which says that $\q_0:=\sigma^{-n+1}(\p)\in \Irr(\q)$. Moreover, $\sigma^{n-1}(\q_0)=\p$ and $n-1\geq 0$ imply that $\q_0\lleq \p$ and $\langle\ca{P}(\q_0, \p)\rangle$ is a proper nonzero submodule of $V_\p$, which is a contradiction. Thus $g\in R^\times$ and $Ra + R\sigma^n(a)=R$. By \cref{R:general-facts-GWA}\cref{R:general-facts-GWA:simple}, $A$ is simple.

Conversely, suppose first that $R$ is a field. Then, for any $\p\mid a$, the module $V_\p={}_R R$ is simple as an $R$-module, hence simple as an $A$-module. So we can assume that $R$ is not a field and that $A$ is simple. In particular, $R$ has no proper nonzero $\sigma$-stable ideals, which by \cref{L:subm-Vp-units} means that $V_1$ is simple, hence of finite length. It follows from the preceding argument and \cref{P:modules-of-finite-len} that all $\mathcal G$-orbits in $\Irr(R)$ are infinite. 

Let $\p\mid a$ and suppose that there exist $\p_0\in\Irr(\p)$ and $\q_0\in\Irr(\q)$ such that $\sigma^{n}(\q_0)=\p_0$, for some $n\in\bb{Z}$. Since $a=\p\sigma^{-1}(\q)$, we conclude that $\p_0=\sigma^{n}(\q_0)$ divides $\sigma^{n+1}(a)=\sigma^{n+1}(\p)\sigma^{n}(\q)$. So, as $\p_0\mid a$, we have
\[
Ra+R\sigma^{n+1}(a)\subseteq\langle \p_0\rangle\subsetneq R.
\]
By the simplicity of $A$ (see \cref{R:general-facts-GWA}\cref{R:general-facts-GWA:simple}), we must have $n+1<1$, i.e.\ $n<0$. Thus, by \cref{cor:simplicity}, $V_\p$ is simple.
\end{proof}

\begin{ex}\label{Ex:kleinian-simple-simple}
Fix $T(h)\in\F[h]$ with $\chara(\F)=0$ and consider the algebra $A$ generated by $x, y, h$ with defining relations
\begin{equation*}
    [h,y] = y, \quad [h,x] = -x, \quad xy=T(h) \quad \text{and}\quad yx=T(h-1),
\end{equation*}
for some $T\in\F[h]$. As argued in \cref{Eg:sl2-prim-quot-as-GWA}, $A$ is a GWA over $\F[h]$ which is simple provided that $T\neq 0$ and that there are no roots $\lambda, \mu$ of $T$ with $\lambda-\mu\in\bb{Z}_{>0}$. This occurs for example in case $T(h)=h^n$ or in case $T(h)=\prod_{i=1}^{n+1}\left(h+i/(n+1)\right)$, for some $n\in\bb{Z}_{>0}$. The latter choice of $T$ corresponds exactly to the algebra of invariants of the Weyl algebra under the action of the cyclic group of order $n+1$ and is known as a \textit{noncommutative deformation of type-$A$ Kleinian singularity} (see \cite{tH93, EKRS21}). Then \cref{thm:all_simple} shows that for the above choices of $T$, all $\F[h]$-free modules of rank $1$ are simple as $A$-modules.
\end{ex}

\subsection{Composition series, length and socles}\label{SS:PID:length}

Unless $R$ is a field, the regular module $V_\p={}_R R$ is not artinian, so it has infinite length as an $R$-module. We want to investigate when $V_\p$ has finite length as an $A$-module. \cref{P:modules-of-finite-len} gives a necessary condition for this to hold: all $\mathcal G$-orbits in $\Irr(R)$ must be infinite. We will show that this condition is also sufficient (compare \cite[Proposition 3.22]{FLM24}).

\begin{thm}
\label{thm:length}
Let $\p$ be a divisor of $a$. Then $V_\p$ has finite length if and only if
all $\mathcal G$-orbits in $\Irr(R)$ are infinite. In this case, the length of $V_\p$ is bounded above by $1+|\Omega_\p|$, where
\begin{align*}
\Omega_\p:=\{(z,w) \in \Irr(\q) \times \Irr(\p) \mid z \lleq w\} 
\end{align*}
with its cardinality $|\Omega_\p|$ considered in the sense of multisets.
\end{thm}
\begin{proof}
As argued above, we just need to consider the case that all $\mathcal G$-orbits in $\Irr(R)$ are infinite. Since in this case the multiset $\Omega_\p$ is finite, it will suffice to prove the upper bound property $\ell(V_\p)\leq 1+|\Omega_\p|$ where $\ell(\cdot)$ denotes length, which we will do by induction on $|\Omega_\p|$.

Note first that the relation $\lleq$ is reflexive, in the sense that if $z \lleq w$ and $w \lleq z$, then $z\sim w$, because we are in the case of infinite $\mathcal G$-orbits. However, some care is needed as we are dealing with multisets $\Irr(\p)$ and $\Irr(\q)$. Suppose, for example, that $\p_0\in\Irr(\p)$ occurs with multiplicity $2$ as an irreducible factor of $\p$. Then there is another element $\p_0'\in\Irr(\p)\setminus\{\p_0\}$ such that $\p_0 \sim \p_0'$, although we consider these to be \textit{distinct} factors. This situation will appear in this proof, where we will write $\p_0\neq \p_0'\sim\p_0$ (and similarly for $\q_0$).

To begin the induction, note that \cref{cor:simplicity} says that $V_\p$ is simple (which means length $1$) if and only if $|\Omega_\p|=0$, so the claim holds in this case. Thus we can assume that $V_\p$ is not simple and that $\ell(V_{\p'})\leq 1+|\Omega_{\p'}|$ for all $\p'$ with $|\Omega_{\p'}|<|\Omega_{\p}|$.

Let $W$ be any maximal submodule of $V_\p$. By \cref{cor:max_subs}, 
\[W=\langle \ca{P}(\q_0,\p_0)\rangle\]
where $(\q_0,\p_0)\in\Omega_{\p}$ with $\p_0=\sigma^n(\q_0)$, for some $n\geq 0$, such that there are no $0\leq i\leq j\leq n$ with $j-i<n$ and $(\sigma^i(\q_0),\sigma^j(\q_0))\in\Omega_{\p}$. For easy reference within this proof, we refer to this as the \textit{minimality condition} for $(\q_0,\p_0)$ in $\Omega_{\p}$. By \cref{L:basic-submodules}, we have $W \simeq V_{\p'}$ where $\p'= \sigma^{-1}(\q_0)\frac{\p}{\p_0}$ and $\q'=\sigma(\p_0)\frac{\q}{\q_0}$. 

We can now set up a partial correspondence between the sets $\Omega_{\p}$ and $\Omega_{\p'}$ noting that, as multisets,
\[\Irr(\p')=(\Irr(\p) \setminus \{\p_0\}) \cup \{\sigma^{-1}(\q_0)\} \quad \text{and}\quad \Irr(\q')=(\Irr(\q) \setminus \{\q_0\}) \cup \{\sigma(\p_0)\},\]
with $\p_0\not\sim\sigma^{-1}(\q_0)$ (equivalently, $\q_0\not\sim\sigma(\p_0)$).

Let $(z,w)\in\Omega_{\p}$. We consider the following cases:
\begin{description}
\item[($z=\q_0$ and $w=\p_0$):] Then the pair $(z,w)$ has no corresponding pair in $\Omega_{\p'}$.
\item[($z\neq\q_0$ and $w\neq\p_0$):] Then $(z,w)\in \Omega_{\p'}$.
\item[($z=\q_0$ and $w\neq\p_0$):] Since $(\q_0,\p_0)$ is minimal in $\Omega_{\p}$ and $(\q_0,w)\in\Omega_{\p}$, it follows that $\p_0\lleq w$. There are two possibilities:
\begin{enumerate}[label=(\roman*)]
\item If $\p_0\llneq w$, then the pair $(\q_0,w)\in\Omega_{\p}$ corresponds to the pair $(\sigma(\p_0), w)$ in $\Omega_{\p'}$.
\item If $\p_0\sim w$, then $(\q_0,w)$ has no corresponding pair in $\Omega_{\p'}$. Thus, if $m\geq 1$ is the multiplicity of $\p_0$ in $\p$, then we loose the $m-1$ pairs of the form $(\q_0, w)$, with $\p_0\neq w\sim \p_0$. 
\end{enumerate}
\item[($z\neq\q_0$ and $w=\p_0$):] This case is entirely analogous to the previous one so we just sketch it. The minimality of $(\q_0,\p_0)$ in $\Omega_{\p}$ implies that $z\lleq \q_0$. There are two possibilities:
\begin{enumerate}[label=(\roman*)]
\item If $z\llneq \q_0$, then the pair $(z,\p_0)\in\Omega_{\p}$ corresponds to the pair $(z, \sigma^{-1}(\q_0))$ in $\Omega_{\p'}$.
\item If $\q_0\sim z$, then we loose the $n-1$ pairs of the form $(z,\p_0)$, with $\q_0\neq z\sim \q_0$, where $n\geq 1$ is the multiplicity of $\q_0$ in $\q$.
\end{enumerate}
\end{description}
It is easy to verify that the above partial correspondence is a bijection onto $\Omega_{\p'}$, so that $|\Omega_{\p'}|=|\Omega_{\p}|-1-(m-1)-(n-1)\leq |\Omega_{\p}|-1$, because $m, n\geq 1$. Hence, using the induction hypothesis, we get
\begin{align*}
\ell(V_\p)&= 1+\ell(V_{\p'})\leq 2+|\Omega_{\p'}|\leq 2+|\Omega_{\p}|-1=1+|\Omega_{\p}|.
\end{align*}
\end{proof}

\begin{ex}\label{Ex:comp-series-with-diag}
Consider the case where $R$ is a PID with infinite $\mathcal G$-orbits in $\Irr(R)$, $z\in\Irr(R)$, $\q=z^2$ and $\p=(\sigma^2(z))^2\sigma^3(z)$, so $a=(\sigma^{-1}(z))^2(\sigma^2(z))^2\sigma^3(z)$. Write $\Irr(\q)=\{\q_0, \q_0'\}$, with $\q_0=z$ and $\q_0'=z$, and $\Irr(\p)=\{\p_0, \p_0', \p_1\}$, with $\p_0=\sigma^2(z)$, $\p_0'=\sigma^2(z)$ and $\p_1=\sigma^3(z)$. 

We illustrate this situation below, where red is used for the factors of $\p$ and blue is used for the factors of $\q$. We see that  $|\Omega_{\p}|=6$.
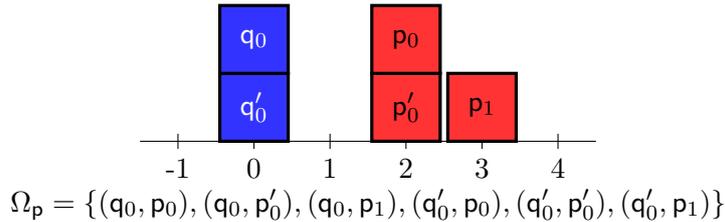
\begin{center}
\begin{tikzpicture}[%
qirr/.style={rectangle, draw=black, fill=blue!80, very thick, minimum size=9mm, text=white}, pirr/.style={rectangle, draw=black, fill=red!80, very thick, minimum size=9mm, text=black}
]
    \draw[-] (-1.5,0) -- (4.5,0) node[below] {};
    \foreach \x in {-1,0,...,4}
        \draw (\x,0.1) -- (\x,-0.1) node[below] {\x};
   \node[qirr] at (0, 0.45) {$\q_0'$};
   \node[qirr] at (0, 1.35) {$\q_0$};
   \node[pirr] at (2, 0.45) {$\p_0'$};
   \node[pirr] at (2, 1.35) {$\p_0$};
   \node[pirr] at (3, 0.45) {$\p_1$};
   \node[align = left, below] at (1.5,-.5) {$\Omega_{\p}=\{(\q_0, \p_0), (\q_0, \p_0'), (\q_0, \p_1), (\q_0', \p_0), (\q_0', \p_0'), (\q_0', \p_1)\}$};
\end{tikzpicture}
\end{center}

Choose the pair $(\q_0, \p_0)$, which defines a maximal submodule $\ca{P}(\q_0,\p_0)$ isomorphic to $V_{\p'}$, where $\p'= \sigma^{-1}(\q_0)\frac{\p}{\p_0}$ and $\q'=\sigma(\p_0)\frac{\q}{\q_0}$. We obtain the following diagram. (Notice that, in the language of the proof of \cref{thm:length}, we have $a=2=b$, so we know from the proof that $|\Omega_{\p'}|=6-a-b+1=3$.)
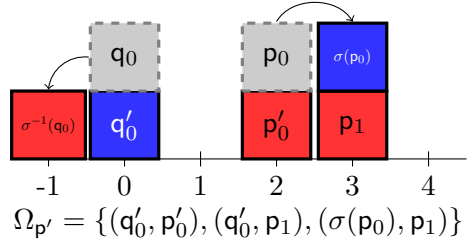
\begin{center}
\begin{tikzpicture}[%
qirr/.style={draw, rectangle, draw=black, fill=blue!80, very thick, minimum size=9mm, text=white}, pirr/.style={draw, rectangle, draw=black, fill=red!80, very thick, minimum size=9mm, text=black}, girr/.style={draw, rectangle, draw=gray, dashed, fill=gray!40, very thick, minimum size=9mm, text=black}
]
    \draw[-] (-1.5,0) -- (4.5,0) node[below] {};
    \foreach \x in {-1,0,...,4}
        \draw (\x,0.1) -- (\x,-0.1) node[below] {\x};
   \node[qirr] at (0, 0.45) {$\q_0'$};
   \node[girr] (qo) at (0, 1.35) {$\q_0$};
   \node[pirr] at (2, 0.45) {$\p_0'$};
   \node[girr] (po) at (2, 1.35) {$\p_0$};
   \node[pirr] at (3, 0.45) {$\p_1$};
   \node[pirr] (lqo) at (-1, 0.45) {{\scalebox{.55}{$\sigma^{-1}(\q_0)$}}};
   \node[align = left, below] at (1.5,-.5) {$\Omega_{\p'}=\{(\q_0', \p_0'), (\q_0', \p_1), (\sigma(\p_0), \p_1)\}$};
   \node[qirr] (rpo) at (3, 1.35) {{\scalebox{.55}{$\sigma(\p_0)$}}};
   \draw[bend right=45,->]  (qo.west) to (lqo.north); 
   \draw[bend left=60,->]  (po.north) to (rpo.north); 
\end{tikzpicture}
\end{center}
Proceeding with the pair $(\q_0', \p_0')\in\Omega_{\p'}$, we obtain the following diagram with $\p''= \sigma^{-1}(\q_0')\frac{\p'}{\p_0'}$.
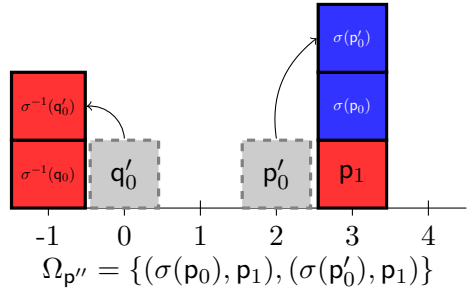
\begin{center}
\begin{tikzpicture}[%
qirr/.style={draw, rectangle, draw=black, fill=blue!80, very thick, minimum size=9mm, text=white}, pirr/.style={draw, rectangle, draw=black, fill=red!80, very thick, minimum size=9mm, text=black}, girr/.style={draw, rectangle, draw=gray, dashed, fill=gray!40, very thick, minimum size=9mm, text=black}
]
    \draw[-] (-1.5,0) -- (4.5,0) node[below] {};
    \foreach \x in {-1,0,...,4}
        \draw (\x,0.1) -- (\x,-0.1) node[below] {\x};
   \node[girr] (qol) at (0, 0.45) {$\q_0'$};
   \node[girr] (pol) at (2, 0.45) {$\p_0'$};
   \node[pirr] at (3, 0.45) {$\p_1$};
   \node[pirr] (lqo) at (-1, 0.45) {{\scalebox{.55}{$\sigma^{-1}(\q_0)$}}};
   \node[pirr] (lqol) at (-1, 1.35) {{\scalebox{.55}{$\sigma^{-1}(\q_0')$}}};
   \node[qirr] (rpo) at (3, 1.35) {{\scalebox{.55}{$\sigma(\p_0)$}}};
    \node[qirr] (rpol) at (3, 2.25) {{\scalebox{.55}{$\sigma(\p_0')$}}};
   \draw[bend right=45,->]  (qol.north) to (lqol.east); 
   \draw[bend left=30,->]  (pol.north) to (rpol.west); 
   \node[align = left, below] at (1.5,-.5) {$\Omega_{\p''}=\{(\sigma(\p_0), \p_1), (\sigma(\p_0'),\p_1)\}$};
\end{tikzpicture}
\end{center}
Finally, taking the pair $(\sigma(\p_0), \p_1)$, we obtain the last diagram with $\p'''= \p_0\frac{\p''}{\p_1}$.
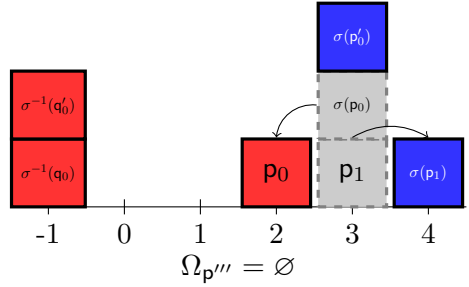
\begin{center}
\begin{tikzpicture}[%
qirr/.style={draw, rectangle, draw=black, fill=blue!80, very thick, minimum size=9mm, text=white}, pirr/.style={draw, rectangle, draw=black, fill=red!80, very thick, minimum size=9mm, text=black}, girr/.style={draw, rectangle, draw=gray, dashed, fill=gray!40, very thick, minimum size=9mm, text=black}
]
    \draw[-] (-1.5,0) -- (4.5,0) node[below] {};
    \foreach \x in {-1,0,...,4}
        \draw (\x,0.1) -- (\x,-0.1) node[below] {\x};
   \node[pirr] (po) at (2, 0.45) {$\p_0$};
   \node[girr] (p1)at (3, 0.45) {$\p_1$};
   \node[pirr] (lqo) at (-1, 0.45) {{\scalebox{.55}{$\sigma^{-1}(\q_0)$}}};
   \node[pirr] (lqol) at (-1, 1.35) {{\scalebox{.55}{$\sigma^{-1}(\q_0')$}}};
   \node[girr] (rpo) at (3, 1.35) {{\scalebox{.55}{$\sigma(\p_0)$}}};
    \node[qirr] (rpol) at (3, 2.25) {{\scalebox{.55}{$\sigma(\p_0')$}}};
    \node[qirr] (rp1) at (4, 0.45) {{\scalebox{.55}{$\sigma(\p_1)$}}};
   \draw[bend right=45,->]  (rpo.west) to (po.north); 
   \draw[bend left=30,->]  (p1.north) to (rp1.north); 
   \node[align = left, below] at (1.5,-.5) {$\Omega_{\p'''}=\varnothing$};
\end{tikzpicture}
\end{center}
It follows that $V_{\p'''}$ is simple and $\ell(V_{\p})=4$, where 
\begin{align*}
V_{\p}\supsetneq \langle g'\rangle \supsetneq \langle g''\rangle\supsetneq \langle g'''\rangle\supsetneq \{0\}
\end{align*}
is a composition series with $\langle g^{(i)}\rangle\simeq V_{\p^{(i)}}$, for $i=1, 2, 3$.
\end{ex}

\begin{rmk}\label{R:algorithm-4-com-ser}
Assume the conditions of \cref{thm:length}.
The proof of this result shows that, in case $a$ is square-free, then we have the equality $\ell(V_\p)=1+|\Omega_\p|$. Moreover, it also indicates an algorithm for producing a composition series of $V_\p$: pick a pair $(\q_0,\p_0)\in\Omega_{\p}$ satisfying the minimality condition for $\q_0\lleq\p_0$. Set $\p'=\sigma^{-1}(\q_0)\frac{\p}{\p_0}$. Then $V_{\p'}$ embeds as a maximal submodule of $V_\p$, which for simplicity we just write as $V_{\p'} \subsetneq V_\p$. Replace $\p$ with $\p'$ and $\q$ with $\q'=\sigma(\p_0)\frac{\q}{\q_0}$ and repeat. This process terminates after at most $|\Omega_\p|$ steps, and we get a composition series (taking isomorphism classes)
\[ \{0\} \subsetneq V_{\p^{(\ell-1)}} \subsetneq \cdots \subsetneq  V_{\p''} \subsetneq  V_{\p'} \subsetneq V_\p.\] 
\end{rmk}

The next result shows that the final configuration for $\p^{(\ell-1)}$ in the algorithm described in \cref{R:algorithm-4-com-ser} is unique and independent of the choices made along the way. See also \cref{Ex:fast-comput-hat-p} for a more direct explanation of this.

\begin{cor}
Assume that all $\mathcal G$-orbits in $\Irr(R)$ are infinite and
let $\p$ be a divisor of $a$. Then $V_\p$ has a unique simple submodule
\[\soc(V_\p) \simeq V_{\hat{\p}},\]
where $\hat{\p}$ can be computed combinatorially as in \cref{R:algorithm-4-com-ser} or \cref{Ex:fast-comput-hat-p} below.
\end{cor}
\begin{proof}
It has already been remarked in \cref{L:radI-subm}\cref{L:radI-subm:d} that $\soc(V_\p)$ is simple, hence of the form $V_{\hat{\p}}$ with $\Omega_{\hat{\p}}=\varnothing$. This is because, $R$ being a PID, any nonzero submodule of $V_\p$ is isomorphic to ${}_RR$ as an $R$-module.
\end{proof}

\begin{ex}\label{Ex:fast-comput-hat-p}
Fix $V_\p$ as above. There is a faster way of obtaining the element $\hat{\p}$ such that $\soc(V_\p) \simeq V_{\hat{\p}}$. Note first that, since $a=\p\sigma^{-1}(\q)$, the irreducible factors of $a$ corresponding to $\q$ are those in $\sigma^{-1}(\Irr(\q))$. Making this shift, the pair $(\p', \sigma^{-1}(\q'))$ for which $V_{\p'}\simeq\langle\ca{P}(\q_0, \p_0)\rangle$ is obtained from the pair $(\p, \sigma^{-1}(\q))$ by moving the factor $\p_0$ from $\p$ to $\sigma^{-1}(\q)$ and the factor $\sigma^{-1}(\q_0)$ from $\sigma^{-1}(\q)$ to $\p$, thus obtaining $(\p', \sigma^{-1}(\q'))$. So, after making the shift to $\sigma^{-1}(\q)$, this process consists just of interchanging irreducible factors. In terms of the diagrams from \cref{Ex:comp-series-with-diag}, this is just switching colors from a minimal pair. This can be done one $\mathcal G$-orbit at a time, so we illustrate this on a single orbit using \cref{Ex:comp-series-with-diag}. Note that now, after the shift, we require of a pair in $\Omega_\p$ that a blue block must be strictly to the left of a red block. So the blocks don't change position, just color, and in the final configuration all the blue blocks must be to the right of the red blocks. There is a unique way to achieve this, as we illustrate now.
\begin{enumerate}[label=\arabic*.]
\item Consider the initial configuration from \cref{Ex:comp-series-with-diag} and shift the blue blocks one unit to the left, to obtain the following diagram.
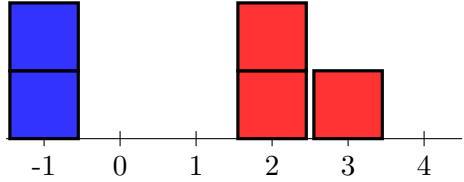
\begin{center}
\begin{tikzpicture}[%
qirr/.style={rectangle, draw=black, fill=blue!80, very thick, minimum size=9mm, text=black}, pirr/.style={rectangle, draw=black, fill=red!80, very thick, minimum size=9mm, text=black}
]
    \draw[-] (-1.5,0) -- (4.5,0) node[below] {};
    \foreach \x in {-1,0,...,4}
        \draw (\x,0.1) -- (\x,-0.1) node[below] {\x};
   \node[qirr] at (-1, 0.45) {};
   \node[qirr] at (-1, 1.35) {};
   \node[pirr] at (2, 0.45) {};
   \node[pirr] at (2, 1.35) {};
   \node[pirr] at (3, 0.45) {};
\end{tikzpicture}
\end{center}

\item Now interchange the colors of the blue and red blocks so that there is no red block to the right of a blue block. There is a unique way of achieving this (ignoring the vertical order, as this doesn't change the prime decomposition of $\hat\p$), depicted below.
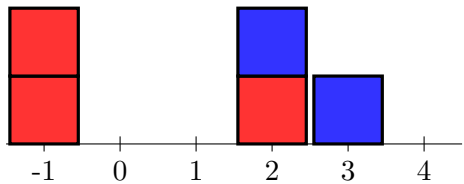
\begin{center}
\begin{tikzpicture}[%
qirr/.style={rectangle, draw=black, fill=blue!80, very thick, minimum size=9mm, text=black}, pirr/.style={rectangle, draw=black, fill=red!80, very thick, minimum size=9mm, text=black}
]
    \draw[-] (-1.5,0) -- (4.5,0) node[below] {};
    \foreach \x in {-1,0,...,4}
        \draw (\x,0.1) -- (\x,-0.1) node[below] {\x};
   \node[pirr] at (-1, 0.45) {};
   \node[pirr] at (-1, 1.35) {};
   \node[pirr] at (2, 0.45) {};
   \node[qirr] at (2, 1.35) {};
   \node[qirr] at (3, 0.45) {};
\end{tikzpicture}
\end{center}

\item This already shows that $\hat\p\sim (\sigma^{-1}(z))^2\sigma^2(z)$. To obtain a diagram which shows the same information as the final diagram in \cref{Ex:comp-series-with-diag} we just need to move the blue blocks back one unit to the right.
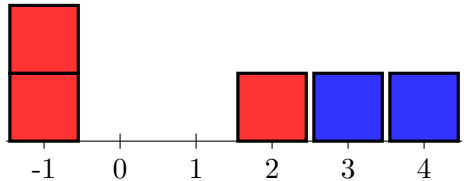
\begin{center}
\begin{tikzpicture}[%
qirr/.style={rectangle, draw=black, fill=blue!80, very thick, minimum size=9mm, text=black}, pirr/.style={rectangle, draw=black, fill=red!80, very thick, minimum size=9mm, text=black}
]
    \draw[-] (-1.5,0) -- (4.5,0) node[below] {};
    \foreach \x in {-1,0,...,4}
        \draw (\x,0.1) -- (\x,-0.1) node[below] {\x};
   \node[pirr] at (-1, 0.45) {};
   \node[pirr] at (-1, 1.35) {};
   \node[pirr] at (2, 0.45) {};
   \node[qirr] at (4, 0.45) {};
   \node[qirr] at (3, 0.45) {};
\end{tikzpicture}
\end{center}
\end{enumerate}
\end{ex}

\noindent Recall that, by definition, $\degg(r)=|\Irr(r)|$, for all nonzero $r\in R$.

\begin{cor}
Assume that all $\mathcal G$-orbits in $\Irr(R)$ are infinite and
let $\p$ be a divisor of $a$.
The length of $V_\p$ is bounded above by $1+\degg(\p) \cdot \degg(\q)$.
\end{cor}

\begin{cor}
Assume that all $\mathcal G$-orbits in $\Irr(R)$ are infinite. 
Then all the objects in $\mathscr{C}_1$ have finite length as $A$-modules and this length is uniformly bounded above by 
\[ 1+\left(\frac{\degg(a)}{2}\right)^2.\]
\end{cor}
\begin{proof}
We have $|\Omega_\p|\leq\degg(\p) \degg(\q)$. Since $\degg(a)=\degg(\p)+\degg(\q)$, the maximum value of $\degg(\p) \degg(\q)$ is attained for $\degg(\p)=\frac{\degg(a)}{2}$ when $\degg(a)$ is even, and for $\degg(\p)=\frac{\degg(a) \pm 1}{2}$ when $\degg(a)$ is odd.
\end{proof}

\begin{ex}
Take $A=R(\sigma, a)$ with $R=\F[h]$, $\chara(\F)=0$ and $\sigma(h)= h+1$. 
\begin{enumerate}[label=(\roman*)]
\item Suppose that $a=h^{2n}$, for some $n\in\bb Z_{\geq 0}$. Then, since $A$ is simple (compare \cref{Ex:kleinian-simple-simple}), it follows from \cref{thm:all_simple} that all $\F[h]$-free modules of rank $1$ are simple, so $V_{h^k}$ has length $1$ for all $0\leq k\leq 2n$.
\item For a contrasting example, suppose that $a=(h-1)^n(h+1)^n=(h^2-1)^{n}$, for some $n\in\bb Z_{> 0}$. Then, e.g.\ using the diagramatics of \cref{Ex:comp-series-with-diag,Ex:fast-comput-hat-p}, we see that $V_{(h+1)^n}$ has length $n+1$.
\end{enumerate} 
\end{ex}

Suppose now instead that some $\mathcal G$-orbit is finite, say $|\orb(z)|=n$, for some $z \in \Irr(R)$. Then all modules $V_\p$ will have infinite length, but we can still describe the possible maximal submodules and corresponding simple subquotients. So let $\langle g \rangle$ be a maximal submodule of $V_\p$. 
As a particular case of \cref{cor:max_subs,T:all-principal-submodules} we have the following, assuming the set-up in this paragraph. 
\begin{enumerate}[label=(\roman*)]
\item Suppose the hypotheses in \cref{cor:max_subs}\cref{cor:max_subs:fin:not-p-and-q} are satisfied. Then with $g= \prod_{i=1}^n \sigma^i(z)$ we have $x \cdot g = \sigma^{-1}(g) \p = ug \p$, where $u=\frac{z}{\sigma^n(z)}\in R^\times$. So it follows that $\langle g \rangle\simeq V_{u\p}\simeq V_{\p}^{\tau_u}$ (see~\cref{L:unit-twists}) and we have an infinite chain of nontrivial inclusions
\[  \cdots \hookrightarrow V_{u^2\p} \hookrightarrow V_{u\p} \hookrightarrow V_\p.\] 
\item Otherwise suppose that $\q_0=\sigma^{j}(z) \in \Irr(\q)$ and $\p_0=\sigma^{k}(z) \in \Irr(\p)$ with $0\leq k-j\leq n-2$ and $\sigma^{i}(z) \not\in \Irr(\p) \cup \Irr(q)$ for any $j<i<k$. Then $g\sim \ca{P}(q_0,p_0)$ generates a maximal submodule isomorphic to $V_{\p'}$ with $\p'= \sigma^{-1}(\q_0)\frac{\p}{\p_0}$.
\end{enumerate} 
This illustrates how to construct filtrations of $V_\p$ even in the infinite length case, but note that as usual, the set of subquotients appearing depends on the choice of filtration and is not uniquely determined by $V_\p$.

\begin{ex}
Let $R=\F[h]$ with $\sigma (h)= \xi h$, for $\xi\in\F$ a primitive $7$-th root of unity, so that $|\mathcal G|=7$. Let $a=(h+1)(\xi^3 h+1)$ and take $\p=\xi^3 h+1$, which gives $\q=\xi h+1$. Then 
$\langle g\rangle=\langle \ca{P}(\q,\p)\rangle = \langle (\xi h+1)(\xi^2 h+1)(\xi^3 h+1)\rangle$ is a maximal submodule of $V_{\p}$ isomorphic to $V_{\p'}$, where $\p'=\sigma^{-1}(\q)=h+1$ and $\q'=\xi^4h+1$. Now $\langle f\rangle=\langle \ca{P}(\q',\p')\rangle = \langle (\xi^4 h+1)(\xi^5 h+1)(\xi^6 h+1)(h+1)\rangle$ is a maximal submodule of $V_{\p'}$ isomorphic to $V_{\p''}$ where $\p''=\sigma^{-1}(\q')=\p$.

In this case we have a $2$-periodic sequence of maximal submodules:
\[  \cdots \hookrightarrow V_{\p'} \hookrightarrow V_\p \hookrightarrow V_{\p'} \hookrightarrow V_\p.\]  
\end{ex}

\subsection{Weight modules and composition factors}\label{SS:PID:weight-mods}

Recall (see e.g.~\cite{DGO96}) that an $A$-module $M$ is a weight module if 
\begin{align*}
M=\bigoplus_{\mathfrak m\in\Max(R)}M_{\mathfrak m}, 
\end{align*}
where $\Max(R)$ is the set of all maximal ideals of $R$ and $M_{\mathfrak m}=\pb{v\in M\mid \mathfrak m \cdot v=0}$. The support of $M$ is $\pb{\mathfrak m\in\Max(R)\mid M_{\mathfrak m}\neq 0}$. We say that $\mathfrak m\in\Max(R)$ is a \textit{break} if $a\in\mathfrak m$.

Presently, our goal is to show that any simple weight module which has finite support will appear as a composition factor of some $V_\p$; and conversely, if $V_\p$ has finite length, then all of its composition factors except for $\soc(V_\p)=V_{\hat\p}$ are simple weight modules with finite support. This shows that any abelian category containing the category of $A$-modules which are free of rank $1$ over $R$ will contain the category of weight modules of finite support, provided that $R$ is not a field (otherwise the concept of a weight module would be vacuous) and that there is at least one nonzero $R$-free module of finite length.

First, suppose that $\p\mid a$ and that $V_\p$ is not simple (so, in particular, $R$ is not a field, by \cref{thm:all_simple}). Let $\langle g\rangle$ be a maximal submodule of $V_\p$. So $0\neq g\notin R^\times$ and by \cref{cor:max_subs} we can write $g=\prod_{i=0}^n \sigma^i(z)$, for some $n\geq 0$ and $z\in\Irr(R)$ with $|\orb(z)|\geq n+1$. Thus, the elements $\sigma^i(z)$ are pairwise coprime and $\mathfrak m_i=\langle \sigma^i(z)\rangle$ is a maximal ideal of $R$, for $0\leq i\leq n$. We have, as $R$-modules,
\begin{align*}
V_\p/\langle g\rangle =R/\langle g\rangle\simeq R/ \mathfrak m_0\oplus \cdots \oplus R/ \mathfrak m_n=\seq{V_\p/\langle g\rangle}_{\mathfrak m_0} \oplus \cdots \oplus \seq{V_\p/\langle g\rangle}_{\mathfrak m_n},
\end{align*}
so $V_\p/\langle g\rangle$ is a weight module with (finite) support equal to $\pb{\mathfrak m_0, \ldots, \mathfrak m_n}$. To conclude, notice that $\langle g\rangle\simeq V_{\p'}$ for some $\p'\mid a$. If $V_{\p'}$ is not simple, we repeat the argument and, in case $V_\p$ has finite length, we stop when we reach $\soc(V_\p)=V_{\hat\p}$, concluding that all composition factors of $V_\p$ except for $V_{\hat\p}$ are weight modules with finite support.

Conversely, let $M$ be a simple weight module with finite support. We assume that $R$ is not a field, as in this case any $A$-module is a weight module with support $\{\langle 0\rangle\}$. Assume also that there exists some nonzero $R$-free $A$-module with a composition series. In particular, by \cref{P:modules-of-finite-len}, $|\orb(z)|=\infty$ for all $z\in\Irr(R)$. Let $\mathfrak m$ be in the support of $M$, so 
$\mathfrak m=\langle z\rangle$, for some $z\in\Irr(R)$. Since $M$ is simple, the support of $M$ is contained in the $\mathcal G$-orbit of $\mathfrak m$ (\cite[Proposition 1.5]{DGO96}). Using the language of~\cite{DGO96}, the $\mathcal G$-orbit of $\mathfrak m$ is linear. Set $\mathfrak m_i=\sigma^i(\mathfrak m)=\langle \sigma^i(z)\rangle$, for every $i\in\bb Z$. Without loss of generality we can assume that the support of $M$ is contained in $\pb{\mathfrak m_i\mid 0\leq i\leq n}$ and contains $\mathfrak m_0, \mathfrak m_n$. Since $xM_{\mathfrak m_i}\subseteq M_{\mathfrak m_{i-1}}$ and $yM_{\mathfrak m_i}\subseteq M_{\mathfrak m_{i+1}}$, it follows from the simplicity of $M$ that $M_{\mathfrak m_i}\simeq R/\mathfrak m_i$ as $R$-modules, for all $0\leq i\leq n$. Moreover, $a(1+\mathfrak m_n)=xy(1+\mathfrak m_n)\subseteq x M_{\mathfrak m_{n+1}}=\langle0\rangle$, so $a\in \mathfrak m_n$; similarly, using $\sigma(a)(1+\mathfrak m_0)=yx(1+\mathfrak m_0)=0$, we conclude that $\sigma(a)\in \mathfrak m_0$. In other words, $\sigma^{-1}(z)\mid a$ and $\sigma^{n}(z)\mid a$. By simplicity of $M$, there is no $0\leq i\leq n-1$ such that $\sigma^{i}(z)\mid a$ (or, in the language of \cite{DGO96}, $\mathfrak m_{-1}$ and $\mathfrak m_{n}$ are consecutive breaks).

Set $\p=\p_0=\sigma^{n}(z)\mid a$ and $\q = \sigma(a/\p)$, so that $a=\p\sigma^{-1}(\q)$, as usual. Take $\q_0=z$. Since $\sigma^{-1}(\q_0)\mid a$ and $\sigma^{-1}(\q_0)$ is coprime with $\p=\sigma^{n}(z)$, it follows that $\sigma^{-1}(\q_0)\mid \sigma^{-1}(\q)$, i.e., $\q_0\mid \q$. Take $g=\ca{P}(\q_0, \p_0)$. Then, by \cref{cor:max_subs}\cref{cor:max_subs:fin:basic}, $\langle g\rangle$ is a maximal submodule of $V_\p$ and $V_\p/\langle g\rangle\simeq R/ \mathfrak m_0\oplus \cdots \oplus R/ \mathfrak m_n$ as $R$-modules. In the case of linear orbits, up to isomorphism, there is a unique simple weight module $M$ with $M_{\mathfrak m_0}\neq 0$, so $M\simeq V_\p/\langle g\rangle$. 

We have thus proved the following result which generalizes \cite[Lemma 3.26, Remark 3.27]{FLM24}.

\begin{thm}
Suppose that $R$ is not a field and that all $\mathcal G$-orbits in $\Irr(R)$ are infinite. Then the following hold:
\begin{enumerate}[label=(\alph*)]
\item All composition factors of any rank $1$ free $R$-module $V_\p$ are simple weight modules with finite support, except for $\soc(V_\p)=V_{\hat\p}$.
\item Any simple weight module with finite support appears as a composition factor of some $V_\p$.
\end{enumerate}
\end{thm}

A connection between weight modules and modules which are free over (the enveloping algebra of) a Cartan subalgebra $\mathfrak h$ already exists in the context of simple Lie algebras $\mathfrak g$, and is given through the so-called \textit{weighting functor}, proposed by Mathieu and used for the first time in the work of Nilsson~\cite{jN16} to classify $\mathfrak g$-modules which are free of finite rank over $U(\mathfrak h)$. It has recently been used in the work of Mendonça (see e.g.~\cite{M25}) in the more general context of $\mathfrak g$-modules which are finitely generated over $U(\mathfrak h)$. 

We note here that an analogous functor can be defined in the GWA setting. Let $\ca{W}: A\text{-}\mathsf{Mod} \ra  A\text{-}\mathsf{Mod}$ be the functor which sends each module $M$ to \[\ca{W}(M) := \bigoplus_{\fr{m} \in \Max(R)}M/\fr{m}M,\] with canonical $R$-action and with
\[x\cdot (v+\fr{m}M)=(x\cdot v+\sigma^{-1}(\fr{m})M) \quad \text{ and } \quad y\cdot (v+\fr{m}M)=(y\cdot v+\sigma(\fr{m})M).\]
Each morphism $f: M \ra N$ is mapped to $\ca{W}(f):\ca{W}(M) \ra \ca{W}(N)$ with $f(v+\fr{m}M)=f(v)+\fr{m}N$.
Here we note that the image of $\ca{W}$ lies in the category of weight modules, with $\ca{W}(M)_{\fr{m}} = M/\fr{m}M$. In particular, for $M\in \mathscr{C}_n$, the module $\ca{W}(M)$ is an everywhere supported weight module with $\dim_{R/\fr{m}} \ca{W}(M)_{\fr{m}} = n$ for each $\fr{m} \in \Max(R)$. In coming work, we plan to investigate this connection between $R$-free modules and weight modules for GWAs in more detail.

\section{Examples} \label{S:EX}
Here we illustrate our results by showing how they apply to a number of different GWAs.

\subsection{Modules for $\fr{sl}_2$}\label{SS:sl2_rank_1}

Consider the Lie algebra $\fr{sl}_2$ over a field $\F$ of characteristic $0$, with basis $e$, $f$, $h$ and Lie brackets $[h,e]=2e$, $[h,f]=-2f$ and $[e,f]=h$. We keep the notation and conventions from \cref{Eg:sl2-as-GWA,Eg:sl2:more}.

The Casimir element $C$ acts as a scalar on each simple $\fr{sl}_2$-module $V$, we call this scalar $\chi$ the central character of the module. Thus, any simple $\fr{sl}_2$-module with a fixed central character $\chi$ is a module for the algebra $\ca{A}_\chi=U(\fr{sl}_2) / \langle C-\chi \rangle$. As discussed in~\cref{Eg:sl2-prim-quot-as-GWA}, this quotient $\ca{A}_\chi$ can still be realized as a generalized Weyl algebra: 
fix a scalar $b \in \F$ and consider the generalized Weyl algebra $R(\sigma,a)$ with 
\[R=\F[h], \quad \sigma(h)=h-2 \quad\text{and}\quad a = -\frac{1}{4}(h+b)(h-b+2).\]
We have an isomorphism 
\[\ca{A}_\chi=  U(\fr{sl}_2)/\langle C-\chi \rangle\simeq R(\sigma, a),\]
where $\chi=\frac{1}{4}b(b-2)$ and the GWA generators $y$ and $x$ correspond to the classes of $e$ and $f$ in the quotient.

Note that all $\mathcal G$-orbits of non-units are infinite since $\chara(\F)=0$ and the units of $R$ are just $\F^\times$, and hence fixed under $\sigma$. It follows that $V_\p\simeq V_{\p'}\iff \p=\p'$, for all divisors $\p, \p'$ of $a$. In particular, if $1\neq\lambda\in\F^\times$, then $V_\p\not\simeq V_{\lambda\p}$.

Using the results of~\cref{SS:C1mods} we note that, for any fixed $b\in\F$, since $\Irr(a)=\{[h+b],[h-b+2]\}$ has size two, then up to unit twists we have four isomorphism classes of $\ca{A}_\chi$-modules which are free of rank $1$ over $R=\F[h]$: 
\[V_1, \quad V_a, \quad V_{h+b}, \quad V_{h-b+2},\]
unless $b=1$, in which case the latter two coincide.

Now, the simplicity of $V_\p$ depends on whether the $\mathcal G$-orbits of the factors of $a$ intersect. We have 
\[h-b+2 \in \orb(h+b) \iff  \left(\exists k\in \bb{Z}\right) h-b+2=\sigma^k(h+b)=h+b-2k \iff b\in \bb{Z}.\] 
So according to \cref{thm:all_simple}, when $b\not\in \bb{Z}$, all four modules $V_\p$ above are simple.

Conversely, assume that $b\in \bb{Z}$ so that $h-b+2 = \sigma^{b-1}(h+b)$. By \cref{cor:simplicity}, $V_{1}$ and $V_{a}$ are both simple. 
For $\p=\p_0=h+b$ we have $\q=\sigma(\frac{a}{\p})\sim \sigma(h-b+2)=h-b=\q_0$, so we have 
\[\sigma^{-b}(\q_0)=\sigma^{-b}(h-b)=h-b+2b=h+b=\p_0,\]
so by \cref{cor:simplicity}, the module $V_{h+b}$ is simple if and only if $b\in\bb{Z}_{>0}$.

So, in conclusion, all modules $V_\p$ are simple except the modules $V_{h-k}$ with $k\in \bb{Z}_{\geq 0}$. We note also that these reducible modules $V_{h-k}$ each have a unique maximal submodule generated by the chain product $(h+k)\cdots (h-k)$, and the quotient is the finite-dimensional simple weight module of dimension $k+1$. 

We have thus recovered most of the results in Section~3 of~\cite{jN15} as a special case of our construction (in \cite{jN15} the element $h$ was scaled by a factor $2$).

\subsection{Modules for deformations of type $A$ Kleinian singularities}
The so-called noncommutative deformations of type-$A$ Kleinian singularities are analogues of the primitive quotients $\ca{A}_\chi$ of $U(\fr{sl}_2)$ discussed above, corresponding to a GWA over $R=\F[h]$ with $a=T(h)=\prod_{i=1}^{n+1}\left(h+i/(n+1)\right)$, for $n\geq 0$. As discussed in \cref{Ex:kleinian-simple-simple}, if $\chara(\F)=0$ then any $\F[h]$-free module of rank $1$ for this GWA is simple and $V_\p\simeq V_{\p'}\iff \p=\p'$, for all divisors $\p, \p'$ of $a$. This means that, up to the unit twists by $\lambda\in\F^\times$, there are $2^{n+1}$ isomorphism classes of such simple free modules of rank $1$. These were also the object of~\cite{FLM24}.

\subsection{Modules for the Weyl algebra ${\bb A}_1$}

To obtain ${\bb A}_1$, which is generated by $x$ and $y$, satisfying the relation $[x, y]=1$, we take $R=\F[\theta]$ with $\sigma(\theta)=\theta-1$ and $a=\theta$. Once more, since $R^\times=\F^\times$ is pointwise fixed by $\sigma$, the isomorphism class of $V_\p$ is determined by $\p\mid a$ and up to the unit twists by $\lambda\in\F^\times$, there are exactly two isomorphism classes of free modules of rank $1$ over $\F[\theta]$. These are in fact the Whittaker modules and their duals (see \cref{SS:whittaker}). Suppose that $\chara(\F)=0$. Then ${\bb A}_1$ is simple, so these modules are also simple by \cref{thm:all_simple}. 

For example, take $\p=1$ and $\q=\sigma(a)=\theta-1$. Consider the basis $(v_k)_{k\geq 0}$ of $\F[\theta]$ defined by $v_k=(\theta-k)\cdots (\theta-1)$. Then relative to this basis we have 
\begin{align*}
x\cdot v_k=v_k+kv_{k-1}\quad\text{and}\quad  y\cdot v_k=v_{k+1}, \quad \text{for all $k\geq 0$.}
\end{align*}
It is readily seen that $V_1$ is the exponential module $\F[\theta]e^\theta$.

\subsection{Modules for the quantum group $\mathrm{GL}_q(2)$}
Let $q\in\F^\times$. A $q$-analogue of the coordinate ring of $2\times 2$ matrices is the algebra $M_q(2)$, given by generators $a, b, c, d$, satisfying the relations
\begin{align*}
ab=qba, \quad ac=qca,\quad bd=qdb,\quad cd=qdc,  \quad
bc= cb \quad \text{and}\quad ad-da = (q- q^{-1})bc.
\end{align*}
(In this example, to preserve the usual notation in the literature, we temporarily use $a$ to denote a generator of the algebra $M_q(2)$ and not the special element $a$ from the definition of a GWA, and use $q$ for a deformation parameter, which should not be confused with the element $\q$ used in the definition of the action of a GWA on a free module of rank $1$.) This algebra contains a special central element $\Delta=ad-qbc=da-q^{-1}bc$ which is usually called the \textit{quantum determinant}. Then the localization of $M_q(2)$ at the powers of $\Delta$ is the quantum group $\mathrm{GL}_q(2)$ of quantum invertible  $2\times 2$ matrices (see e.g.~\cite{T02, BG02}).

Consider $R=\F[b, c, \Delta^{\pm 1}]$ and its automorphism $\sigma$ defined by $\sigma(b)=q^{-1}b$, $\sigma(c)=q^{-1}c$ and $\sigma(\Delta)=\Delta$. Then $\mathrm{GL}_q(2)$ is the GWA $R(\sigma, \Delta+qbc)$. Assume henceforth that $q$ is not a root of unity, so that $\mathcal G$ is infinite. Then $\cent{\mathrm{GL}_q(2)}=R^\sigma=\F[\Delta^{\pm 1}]$. Since $\Irr(R)$ has finite orbits, there will be no simple $\mathrm{GL}_q(2)$-modules that are free over $R$. 

It is natural to consider primitive quotients of $\mathrm{GL}_q(2)$. If we further assume that $\F$ is algebraically closed, then the primitive ideals of $\mathrm{GL}_q(2)$ are (see~\cite[11.8.7]{BG02}):
\begin{align*}
\mathrm{(I)}\  \langle a-\lambda, b, c, d -\mu \rangle;\quad
\mathrm{(II)}\  \langle b, ad-\lambda \rangle,\ \langle c, ad-\lambda \rangle;\quad
\mathrm{(III)}\ \langle b-\lambda c, \Delta-\mu\rangle;
\end{align*}
where $\lambda, \mu\in\F^\times$. The representations of $\mathrm{GL}_q(2)$ that factor through the ideals of type (I) are one dimensional, and those that factor through the ideals of type (II) are representations of the quantum plane defined by $ac=qca$ and localized at $a$, which is a GWA over the PID $\F[c]$; finally, the representations that factor through an ideal of the form $\langle b-\lambda c, \Delta-\mu\rangle$ are representations of the GWA $\F[c](\overline\sigma, q\lambda c^2+\mu)$, also defined over the PID $\F[c]$, with $\overline\sigma(c)=q^{-1}c$. We can write $q\lambda c^2+\mu=q\lambda (c-\xi)(c+\xi)$ with $\xi\in\F^\times$. In case $\chara(\F)\neq 2$, $\xi\neq -\xi$ so $\Irr(q\lambda c^2+\mu)=\{[c-\xi], [c+\xi]\}$ has size two and, up to unit twists, there are four isomorphism classes of modules for this primitive quotient of $\mathrm{GL}_q(2)$ which are free of rank one over $\F[c]$: $V_1$, $V_{c-\xi}$, $V_{c+\xi}$ and $V_{(c-\xi)(c+\xi)}$. The orbits of $c-\xi$ and $c+\xi$ under $\overline\sigma$ do not intersect, as $q$ is not a root of unity, which shows that the only nontrivial submodules of these $V_\p$ are given by the finite orbits of $\overline\sigma$ on $\Irr(\F[c])$. The unique such is $\orb(c)$ and if we view $\langle c\rangle$ as a submodule of $V_\p$ then $\langle c\rangle\simeq V_{q\p}$ with $\dim_\F V_\p/\langle c\rangle=1$. Hence, for any divisor $\p$ of $q\lambda c^2+\mu$ in $\F[c]$, we get an infinite chain of inclusions 
\[  \cdots \hookrightarrow V_{q^2\p} \hookrightarrow V_{q\p} \hookrightarrow V_\p.\]

\section{$R$-free modules of higher rank} \label{SEC:highrank}
In this section, we investigate $R$-free modules of arbitrary finite rank; the objects of $\mathscr{C}_n$. We continue to assume that $R$ is a commutative integral domain and that $a\neq 0$.

\subsection{Stratification of modules in $\mathscr{C}_n$}

If $B=(b_{ij})\in \mathrm{Mat}_{m\times n}(R)$ is a matrix or coordinate-vector, we can apply $\sigma$ entry-wise, and we write just $\sigma(B)$ for $(\sigma(b_{ij}))$. We note that $\sigma(BC)=\sigma(B)\sigma(C)$, $\sigma(\det(B))=\det(\sigma(B))$, and $(\sigma(B))^{-1}=\sigma(B^{-1})$, whenever
the corresponding operations are defined.

Let $V$ be an $R$-free module of rank $n$ and suppose that the action of $R$ on $V$ extends to $A$. Without loss of generality we may assume that $V=R^n$ as an $R$-module. Let $e_1,\ldots, e_n \in R^n$ be the standard $R$-basis vectors, and let $\pi_1,\ldots, \pi_n: R^n \ra R$ be the canonical projections.  Set $p_{ij}:=\pi_i(x \cdot e_j)$ and $\bigp := (p_{ij}) \in \mathrm{Mat}_n(R)$, and similarly,
$q_{ij}:=\pi_i(y \cdot e_j)$, and $\bigq := (q_{ij}) \in \mathrm{Mat}_n(R)$. Then the relations in $A$ and the commutativity of $R$ imply that the actions of $x$ and $y$ on general elements of $V$ are given by
\begin{align}\label{E:bigaction}
x \cdot v = \bigp \sigma^{-1}(v) \quad\text{ and }\quad y \cdot v = \bigq \sigma(v), 
\end{align}
where $v\in R^n$ is written as a column matrix as usual. Moreover, the relation $x\cdot (y \cdot v) = av$ implies that the matrices $\bigp$ and $\bigq$ must be \textit{compatible} in the sense that $\bigp\sigma^{-1}(\bigq)=aI$. Similarly, the relation $y\cdot (x \cdot v) = \sigma(a)v$ gives $\bigq\sigma(\bigp)=\sigma(a)I$. 

Let $\mathrm{Frac}(R)$ be the field of fractions of $R$. The two last matrix equations show in particular that $\bigp$ and $\sigma^{-1}(\bigq)$ commute and that $\bigp$ and $\bigq$ are invertible when viewed as elements in $\mathrm{Mat}_n(\mathrm{Frac}(R))$, since $a, \sigma(a)\neq 0$. We deduce that $\bigq=\sigma(a\bigp^{-1})$ and that $a\bigp^{-1}\in \mathrm{Mat}_n(R)$; in turn, the latter imply that $\bigp\sigma^{-1}(\bigq)=aI$ and $\bigq\sigma(\bigp)=\sigma(a)I$.

Conversely, any matrix $\bigp \in \mathrm{Mat}_n(R)$ such that $a\bigp^{-1}\in \mathrm{Mat}_n(R)$ gives rise to an $A$-module $V$ which is free of rank $n$ over $R$, by setting $\bigq=\sigma(a\bigp^{-1})$ and using~\cref{E:bigaction}. As in the case of rank $1$, we shall write $V=V_{\bigp}$ and, for $\bigp$ fixed, we shall always write $\bigq$ for $\bigq_{\bigp}=\sigma(a\bigp^{-1})$.

In summary, we have shown the following.

\begin{prop}
Let $\bigp \in \mathrm{Mat}_n(R)$ with $\det(\bigp)\neq 0$, such that $a\bigp^{-1}\in \mathrm{Mat}_n(R)$, and set $\bigq:=\sigma(a\bigp^{-1})$.
Let $V_\bigp=R^n$, as an $R$-module, and for $v\in V_\bigp$ define
\[x \cdot v = \bigp\sigma^{-1}(v) \quad \text{ and } \quad y \cdot v = \bigq\sigma(v).\]
Under these actions, $V_P$ is an $A$-module in $\mathscr{C}_n$. Moreover, any $A$-module that is $R$-free of finite rank is isomorphic to some $V_\bigp$. 
\end{prop}

\begin{lemma}\label{L:rankn:isos}
We have \[V_\bigp \simeq V_{\bigp'} \quad \text{ if and only if } \quad \bigp'=S\bigp\sigma^{-1}(S^{-1})\] for some invertible $S\in\mathrm{Mat}_n(R)$.
\end{lemma}
\begin{proof}
An $A$-module isomorphism $\varphi: V_\bigp \ra V_{\bigp'}$ is also an isomorphism of $R$-modules, so $V_\bigp$ and $V_{\bigp'}$ are free of the same rank $n$, and $\varphi(v)=Sv$ for all $v\in R^n$ for some invertible matrix $S \in \mathrm{Mat}_n(R)$.
Since $\varphi$ is a homomorphism of $A$-modules, 
\[S\bigp\sigma^{-1}(v)=\varphi(\bigp\sigma^{-1}(v))= \varphi(x \cdot v) 
=x\cdot\varphi(v)= x\cdot Sv = \bigp'\sigma^{-1}(S)\sigma^{-1}(v)\]
 for all $v\in V_\bigp$, and hence $S\bigp = \bigp'\sigma^{-1}(S)$, from which the lemma follows.
\end{proof}

Next, we show how the Smith normal form provides a stratification of $\mathscr{C}_n$. A corresponding analysis for the $\fr{sl}_2$-case can be found in Section~3 of~\cite{MP17}. We start by recalling the following facts about the Smith normal form.

\begin{lemma}
Let $R$ be a PID and let $M \in \mathrm{Mat}_{m\times n}(R)$. Then there exist $S \in \mathrm{GL}_m(R)$, $T \in\mathrm{GL}_n(R)$, and a quasi-diagonal matrix $D \in \mathrm{Mat}_{m\times n}(R)$ such that
\[M=SDT.\]
where $D=\mathrm{diag}(d_1,d_2, \ldots, d_k,0, \ldots ,0)$ with nonzero diagonal elements $d_i$ satisfying $d_i|d_{i+1}$. The $d_i$ are called the {\bf invariant factors} and they are uniquely determined by $M$ up to units. The matrix $D$ (or the lists of its diagonal elements) is called the {\bf Smith normal form} (SNF) of $M$.

Additionally, the invariant factors satisfy
\[d_i=\frac{g_{i}}{g_{i-1}}\]
where $g_i$ is the greatest common divisor of all $i\times i$-minors of $M$ (with $g_0:=1$).
\end{lemma}

\begin{lemma}\label{L:rankn:dn-div-a}
Let $R$ be a PID and let $\bigp \in \mathrm{Mat}_n(R)$. Then $\bigp$ defines a module $V_\bigp \in \mathscr{C}_n$ if and only if the $n$'th invariant factor of $\bigp$ divides $a$.
\end{lemma}
\begin{proof}
Since $R$ is a PID, the Smith normal form allows us to find invertible matrices $S,T \in \mathrm{Mat}_n(R)$ such that $\bigp=SDT$ with $D=\mathrm{diag}(d_1, \ldots, d_n)$ and $d_i|d_{i+1}$, where the invariant factors $d_i$ are uniquely determined by $\bigp$ up to multiplication by units of $R$.
Then in $\mathrm{Mat}_n(\mathrm{Frac}{(R)})$ we have $\bigq=\sigma(a\bigp^{-1}) = \sigma(T^{-1}aD^{-1}S^{-1})$. Since $D^{-1}=(d_1^{-1}, \ldots, d_n^{-1})$, the matrix $\bigq$ has all coefficients in $R$ if and only if $ad_n^{-1} \in R$, or simply $d_n|a$. 
\end{proof}
The argument above also shows that \[\mathrm{SNF}(\bigp)=(d_1,\ldots,d_n)\; \Leftrightarrow \; \mathrm{SNF}(\bigq)=(\sigma(\tfrac{a}{d_n}),\ldots, \sigma(\tfrac{a}{d_1})).\]

\begin{cor}
Let $d=(d_1,\ldots,d_n) \in R^n$ such that $d_i|d_{i+1}$ and $d_n|a$, and let $D=\mathrm{diag}(d_1,\ldots, d_n)$. Define maps 
\[\Phi_d: \mathrm{GL}_n(R) \times  \mathrm{GL}_n(R) \rightarrow 
A\text{-}\mathsf{Mod} \text{ \ and \ } \Xi_d: \mathrm{GL}_n(R) \rightarrow A\text{-}\mathsf{Mod}\]
by
\[\Phi_d(S,T)=V_{SDT} \text{ and } \Xi_d(S)=V_{DS}.\]
Then, for any $A$-module $V_\bigp$ we have $V_\bigp=\Phi_d(S,T)$ for some $S,T\in \mathrm{GL}_n(R)$ and some $d$ uniquely determined by $\bigp$.
Similarly, any $A$-module that is $R$-free of rank $n$ is isomorphic to some $\Xi_d(S)$.
\end{cor}
\begin{proof}
Writing $\bigp=SDT$ in Smith normal form, we have $V_\bigp = V_{SDT}=\Phi_d(S,T)$, where $d=(d_1,\ldots,d_n)$ is formed from the diagonal entries of $D$ and $d_n|a$ by \cref{L:rankn:dn-div-a}. For the second claim, take $T'=T\sigma^{-1}(S)$. Then by \cref{L:rankn:isos} we have
\[V_\bigp = V_{SDT} \simeq V_{S^{-1}(SDT)\sigma^{-1}(S)}=V_{DT'}=\Xi_d(T').\]
\end{proof}

So the SNF of $\bigp$ provides a stratification on the isoclasses of $R$-free $A$-modules of a fixed finite rank, but note that even if $\mathrm{SNF}(\bigp)=\mathrm{SNF}(\bigp')$, the modules $V_\bigp$ and $V_{\bigp'}$ may be non-isomorphic.

\subsection{Construction of simple modules of arbitrary finite rank}
We shall show that, under appropriate assumptions, simple modules over $A$ of arbitrary finite rank over $R$ exist, by giving explicit families of examples. 

We shall need the following lemma on translated multisets of integers. When $S$ is an integer multiset and $n\in \bb{Z}$, we define the multiset $S-n:=\{s-n \mid s\in S\}$.
\begin{lemma}
\label{lemma:set_equation}
For fixed integers $n,j,k$, the multiset equation
\begin{equation}
\label{eq:set_equation}
\{j\} \cup (S-n) \cup T = \{k\} \cup S \cup (T-n)
\end{equation}
has a finite solution $S,T$ with $|S|,|T|<\infty$ if and only if $n|j-k$.
\end{lemma}
\begin{proof}
First, assume that $n|j-k$. By symmetry in \cref{eq:set_equation} we may assume that $j \leq k$. Then $(S,T)=(\varnothing,\{j+n,j+2n, \ldots, k\})$ solves the equation.

For the reverse implication, we can assume that $n\neq 0$ as this case is trivial. We introduce the following notation: given a finite multiset of integers $E$, let $f_E\in\bb{Z}[t]$ be the monic polynomial whose multiset of roots is equal to $E$, so that $f_E(t)=\prod_{j\in E}(t-j)$. Note that $f_\varnothing=1$ and $f_E=f_{\widetilde E}\iff E=\widetilde E$; furthermore, $f_{E-n}(t)=f_E(t+n)$ for all $n\in\bb{Z}$. So suppose that $S$ and $T$ are finite multisets of integers satisfying \cref{eq:set_equation}. This is equivalent to the following equation in $\bb{Z}[t]$:
\begin{align*}
(t-j)f_S(t+n)f_T(t)=(t-k)f_S(t)f_T(t+n).
\end{align*}
Reducing the coefficients modulo $n$ we get $(t-j)f_S(t)f_T(t)\equiv(t-k)f_S(t)f_T(t)$, which gives $(j-k)f_S(t)f_T(t)\equiv 0$ in $\bb{Z}_n[t]$. As $f_S(t)f_T(t)$ is monic, it follows that $n|j-k$.
\end{proof}

The following is the main result of this section and one of the main results of the paper as it constructs simple $R$-free $A$-modules of arbitrary finite rank under some mild restrictions. Note that the very existence of a simple $R$-free $A$-module implies that $\orb([z])$ is infinite for all $z\in\Irr(R)$, by \cref{P:modules-of-finite-len}.

\begin{thm}\label{T:main:rank:n}
\label{P:simple_rank_n}
Let $R$ be a PID and assume that $a$ is not a unit and that all $\mathcal G$-orbits in $\Irr(R)$ are infinite. Let $a_0$ be any irreducible factor of $a$ that is minimal among the factors of $a$, in the sense that $\sigma^{-k}(a_0)$ is not a factor of $a$ for any $k>0$. 
For arbitrary $n\geq 1$, define an action of $A$ on $V_n(a_0):=R^n$, where $R$ acts by left multiplication and
\begin{align}
x \cdot r e_i &= \sigma^{-1}(r) e_{i+1} \; \text{for $i<n$, and }\; x \cdot r e_n = a_0\sigma^{-1}(r)e_1 \label{eq:x_act_rank_n}, \\
y \cdot r e_i &= \sigma(ar)e_{i-1} \; \text{for $i>1$, and }\; y \cdot r e_1 = \sigma(\tfrac{ra}{a_0})e_n. \label{eq:y_act_rank_n} 
\end{align}
Under this action, $V_n(a_0) \in \mathscr{C}_n$ is a simple $A$-module which is free of rank $n$ over $R$.
\end{thm}
\begin{proof}
We note that $V_n(a_0) = V_\bigp$ where the matrices $\bigp$ and $\bigq=\sigma(a\bigp^{-1})$ in $\mathrm{Mat}_n(R)$ are given by
\[
\bigp = \begin{pmatrix}
0 & 0 & 0 & \cdots & 0 & a_0 \\
1 & 0 & 0 & \cdots & 0 & 0 \\
0 & 1 & 0 & \cdots & 0 & 0 \\
0 & 0 & 1 & \ddots & 0 & 0 \\
\vdots & \vdots & \ddots & \ddots & \vdots & \vdots \\
0 & 0 & 0 & \cdots & 1 & 0
\end{pmatrix} \qquad \bigq=\begin{pmatrix}
0 & \sigma(a) & 0 & 0 & \cdots & 0 \\
0 & 0 & \sigma(a) & 0 & \cdots & 0 \\
0 & 0 & 0 & \sigma(a) & \ddots & 0 \\
\vdots & \vdots & \vdots & \ddots & \ddots & \vdots \\
0 & 0 & 0 & \cdots & 0 & \sigma(a) \\
\sigma(\tfrac{a}{a_0}) & 0 & 0 & \cdots & 0 & 0
\end{pmatrix}.
\]
Note that $\bigp=\mathrm{Comp}(t^n-a_0)$, the companion matrix of the polynomial $p(t)=t^n-a_0 \in R[t]$. 


By \cref{eq:x_act_rank_n} we see that $x^n$ acts diagonally on $V_\bigp$: 
\[x^n \cdot v = \begin{pmatrix}
a_0&&&\\
&\sigma^{-1}(a_0)&&\\
&&\ddots&\\
&&&\sigma^{-(n-1)}(a_0)\\
\end{pmatrix}\sigma^{-n}(v) = \begin{pmatrix}
a_0\sigma^{-n}(v_1)\\
\sigma^{-1}(a_0)\sigma^{-n}(v_2)\\
\vdots\\
\sigma^{-(n-1)}(a_0)\sigma^{-n}(v_n)\\
\end{pmatrix}.\]

Now let $w=(w_1,\ldots, w_n) \in V_\bigp$ be nonzero and let $W=A\cdot w$ be the submodule of $V_\bigp$ generated by $w$. We shall prove that $W=R^n$. First we show that $W$ contains some element of form $re_i$ by proving that if $w$ is not already of the this form, then $W$ contains a nonzero element with more zero coordinates than $w$. Fix an index $k$ for which $w_k \neq 0$ and note that $W$ contains the element
\[w' = w_k\cdot (x^n  \cdot w) - \sigma^{-n}(w_k)\sigma^{-(k-1)}(a_0)\cdot w.\]
The $j$'th coordinate of $w'$ is 
\[w_j'=\pi_j(w') = w_k \sigma^{-(j-1)}(a_0)\sigma^{-n}(w_j) - \sigma^{-n}(w_k)\sigma^{-(k-1)}(a_0)w_j.\]
So $w'_k=0$ and if $w_j=0$ then $w'_j=0$, showing that $w'$ has more zero coordinates than $w$. It remains to show that $w' \neq 0$. Assume that $j\neq k$ is an index for which $w_j\neq 0$ but $w'_j=0$. Then
\begin{equation}
\label{eq:make_zeros}
w_k \sigma^{-(j-1)}(a_0)\sigma^{-n}(w_j) = \sigma^{-n}(w_k)\sigma^{-(k-1)}(a_0)w_j.
\end{equation}
Factoring each side of \cref{eq:make_zeros} into irreducibles, we note that equality holds if and only if it holds in each $\mathcal G$-orbit of irreducibles. In particular, equation~\cref{eq:make_zeros} must hold for the factors in $\orb(a_0)$, so without loss of generality we may assume that all factors of $w_k$ and $w_j$ are in the same orbit as $a_0$. We can then identify (up to units) $w_k$ and $w_j$ with integer multisets $S$ and $T$ such that
\[w_k \sim \prod_{s \in S} \sigma^{s}(a_0) \text{ and } w_j \sim \prod_{t \in T} \sigma^{t}(a_0).\]
Substituting this into \cref{eq:make_zeros} we get
\[ \big(\prod_{s \in S} \sigma^{s}(a_0)\big) \sigma^{-(j-1)}(a_0)\sigma^{-n}\big(\prod_{t \in T}\sigma^{t}(a_0)\big) = \sigma^{-n}\big(\prod_{s \in S} \sigma^{s}(a_0)\big)\sigma^{-(k-1)}(a_0)\big(\prod_{t \in T}\sigma^{t}(a_0)\big),\]
or equivalently
\[ \big(\prod_{s \in S} \sigma^{s}(a_0)\big) \sigma^{-(j-1)}(a_0)\big(\prod_{t \in T}\sigma^{t-n}(a_0)\big) = \big(\prod_{s \in S} \sigma^{s-n}(a_0) \big)\sigma^{-(k-1)}(a_0)\big(\prod_{t \in T}\sigma^{t}(a_0)\big).\]
Since the $\mathcal G$-orbit of $a_0$ is infinite we can compare the exponents of the various $\sigma^{i}(a_0)$. This yields the finite integer multiset-equation
\[S \cup \{1-j\} \cup (T-n) = (S-n) \cup \{1-k\} \cup T. \]
But by \cref{lemma:set_equation}, this equation has no solution since $0 < |j-k| <n$.
We have shown that we can construct $w'\in W$ with exactly one more zero coordinate than $w$, so repeating this process we see that the $A$-submodule $W \subseteq V_\bigp$ contains a nonzero element of form $re_i$. Here we may assume that $i=1$, otherwise replace $re_i$ by $x^{n+1-i} \cdot re_i$. 

Now let $A'$ be the subalgebra of $R(\sigma,a)$ generated by $R$, $x^n$ and $y^n$. Then \[A' \simeq R(\sigma^n,a\sigma^{-1}(a)\cdots \sigma^{-(n-1)}(a)) = R(\sigma',a')\] so $A'$ is also a generalized Weyl algebra and $Re_1 \subseteq V_\bigp$ is an $R$-free module of rank $1$ over $A'$. By \cref{eq:x_act_rank_n} we see that $Re_1 \simeq_{A'} V_{\p}$ with $\p=a_0$. 
But if $a=a_0 a_1 \cdots a_m$ is a factorization of $a$ into irreducibles, then
\[a'=\prod_{j=0}^m\prod_{i=0}^{n-1}\sigma^{-i}(a_j).\]
We show that $a_0$ is a minimal element (in the sense of the statement of the theorem) with respect to the $\sigma'$ orbits of all irreducible factors of $a'$. 
Assume for the sake of contradiction that $(\sigma')^{-k}(a_0) \sim \sigma^{-i}(a_j)$ for some $0 \leq i \leq n-1$, $0 \leq j \leq m$ and $k>0$.
This implies that $\sigma^{-(nk-i)}(a_0) \sim a_j$ with $nk-i \geq n-i > 0$, which contradicts the stated assumption of the minimality of $a_0$ in the $\mathcal G$-orbits of factors of $a$. 

Now applying the simplicity criterion of \cref{cor:simplicity} to the $A'$-module $V_\p=V_{a_0}$, we obtain that $V_\p\simeq Re_1$ is simple. Therefore $b(re_1)=e_1$ for some $b\in A' \subseteq A$. This shows that $e_1 \in W$, which implies that $x^{k} \cdot e_1 = e_{k+1} \in W$ for $0 \leq k \leq n-1$, and hence $W=R^n$ since $W$ is an $R$-module.

Thus we have shown that $\{0\}$ and $R^n$ are the only submodules of $V_\bigp$, so $V_\bigp$ is simple, as claimed.
\end{proof}

\begin{rmk}\label{R:rankn:several-minimals}
The last invariant factor of the matrix $\bigp$ in the proof of \cref{T:main:rank:n} is the divisor $a_0$ of $a$, hence distinct choices of $a_0$ (up to associates) produce non-isomorphic modules. Moreover, 
note that \cref{T:main:rank:n} and its proof still hold if we replace $a_0$ with a product of minimal (in the sense of the theorem) factors of $a$ from distinct $\mathcal G$-orbits.
\end{rmk}

We apply \cref{P:simple_rank_n} above to construct new families of simple $\fr{sl}_2$-modules of rank $n$ over $\F[h]$.
As in~\cref{SS:sl2_rank_1}, let
\[R=\F[h], \quad \sigma(h)=h-2, \quad a = -\frac{1}{4}(h+b)(h-b+2).\]
Then $R(\sigma,a) \simeq U(\fr{sl}_2) / \langle C - \chi\rangle$, with $\chi=\frac{1}{4}b(b-2)$.

If $\chara(\F)=0$ and $b \not\in \bb{Z}$, the $\mathcal G$-orbits of $\Irr(a)$ are non-intersecting so, for arbitrary $n>0$, \cref{P:simple_rank_n} says that the module $V_n(h+b)$ is a simple $\fr{sl}_2$-module of rank $n$ with central character $\chi$. 
We can state this more explicitly as follows.

\subsection{Application to $\fr{sl}_2$}
Consider the Lie algebra $\fr{sl}_2$ with its standard basis $\{e,f,h\}$, over a field $\F$ of characteristic $0$.
For $b\in \F$ and $n\in \bb{Z}_{>0}$, let $V_n^{(b)}=V_n(h+b)=\F[t]^n$ and define an action of $\fr{sl}_2$ on $V_n^{(b)}$ by

\[
h \cdot \begin{pmatrix}f_1(t)\\f_2(t)\\ \vdots \\ \vdots \\ \vdots \\ f_n(t)\end{pmatrix}=\begin{pmatrix}tf_1(t)\\tf_2(t)\\ \vdots \\ \vdots \\ \vdots \\ tf_n(t)\end{pmatrix}, 
\qquad 
e \cdot \begin{pmatrix}f_1(t)\\f_2(t)\\ \vdots \\ \vdots \\ \vdots \\ f_n(t)\end{pmatrix} =
-\frac{1}{4} \begin{pmatrix}
0 & \theta & 0 & 0 & \cdots & 0 \\
0 & 0 & \theta & 0 & \cdots & 0 \\
0 & 0 & 0 & \theta & \ddots & 0 \\
\vdots & \vdots & \vdots & \ddots & \ddots & \vdots \\
0 & 0 & 0 & \cdots & 0 & \theta \\
t-b & 0 & 0 & \cdots & 0 & 0
\end{pmatrix}
\begin{pmatrix}f_1(t-2)\\f_2(t-2)\\ \vdots \\ \vdots \\ \vdots \\ f_n(t-2)\end{pmatrix}
\]
\[\text{and} \qquad
f \cdot \begin{pmatrix}f_1(t)\\f_2(t)\\ \vdots \\ \vdots \\ \vdots \\ f_n(t)\end{pmatrix} = \begin{pmatrix}
0 & 0 & 0 & \cdots & 0 & t+b \\
1 & 0 & 0 & \cdots & 0 & 0 \\
0 & 1 & 0 & \cdots & 0 & 0 \\
0 & 0 & 1 & \ddots & 0 & 0 \\
\vdots & \vdots & \ddots & \ddots & \vdots & \vdots \\
0 & 0 & 0 & \cdots & 1 & 0
\end{pmatrix}
\begin{pmatrix}f_1(t+2)\\f_2(t+2)\\ \vdots \\ \vdots \\ \vdots \\ f_n(t+2)\end{pmatrix},
\]
where $\theta:=(t+b-2)(t-b)=t^2-2t-b(b-2)$.

Under this action, except if $b\in\bb{Z}_{\leq 0}$, the irreducible factor $h+b$ of $a$ is minimal, whence \cref{P:simple_rank_n} guarantees that $V_n^{(b)}$ is a simple $\fr{sl}_2$-module with \[\rank \; \mathrm{Res}_{\F[h]}^{U(\fr{sl}_2)} V_n^{(b)} = n.\] In other words,  $V_n^{(b)}$ is simple and free of rank $n$ when restricted to the subalgebra $\F[h] = U(\fr{h}) \subseteq U(\fr{sl}_2)$.
The central character of $V_n^{(b)}$ is given by the action of the Casimir operator as the scalar $\chi=\frac{1}{4}b(b-2)$. 
In case $b\in\bb{Z}_{\leq 0}$, just replace $b$ with $2-b$ in the discussion above. Moreover, the modules $V_n^{(b)}$ defined herein are non-isomorphic to the modules constructed in~\cite{MP17}, as $V_n^{(b)}$ corresponds to the Smith type $\mathrm{SNF}(\bigp)=(1, \ldots, 1, t+b) \neq (1, \ldots, 1)$, as seen from the action of $f$ above. 

Finally, notice that \cref{R:rankn:several-minimals} allows for a further generalization in case $b \not\in \bb{Z}$. In this situation, the two irreducible factors of $a$, namely $h+b$ and $h-b+2$, are in distinct $\mathcal G$-orbits so they are both minimal. Then, for any $n>0$, $V_n((h+b)(h-b+2))$ is a simple $\fr{sl}_2$-module which is free of rank $n$ over $\F[h]$ and has Smith type $(1, \ldots, 1, (t+b)(t-b+2))$.

\section*{Acknowledgements}
The first named author was partially supported by CMUP -- Centro de Matem\'atica da Universidade do Porto, member of LASI, which is financed by national funds through FCT -- Funda\c c\~ao para a Ci\^encia e a Tecnologia, I.P., under the project with reference UID/00144/2025, doi: \url{https://doi.org/10.54499/UID/00144/2025}.

This work was initiated during a visit of the first named author to Link\"oping University in 2024 and he gratefully acknowledges the warm hospitality extended during this stay as well as the organizers of the VIII International Workshop on Non-Associative Algebras in Link\"oping, which motivated the visit.

The latter part of this work was carried out during a visit of the second named author to the first named author at the University of Porto in November 2025, and he gratefully acknowledges the warm hospitality extended during this stay. This visit was supported by a stipend from SVeFUM - \emph{Stiftelsen för Vetenskaplig Forskning och Utbildning i Matematik}, whose financial assistance is gratefully acknowledged.

\newcommand{\germ}{\mathfrak}
\bibliographystyle{plain}
\def\cprime{$'$}

\end{document}